\documentclass[11pt]{amsart}
\usepackage{amsmath}
\usepackage[active]{srcltx}
\usepackage{t1enc}
\usepackage{cite}
\usepackage[latin2]{inputenc}
\usepackage{verbatim}
\usepackage{amsmath,amsfonts,amssymb,amsthm}
\usepackage[mathcal]{eucal}
\usepackage{enumerate}
\usepackage[centertags]{amsmath}
\usepackage{graphics}
\usepackage{refcount}

\newtheorem{theorem}{Theorem}
\newtheorem{prop}{Proposition}

\newtheorem{lemma}{Lemma}
\newtheorem{remark}{Remark}
\newtheorem{definition}{Definition}

\newtheorem*{MSW}{Theorem MSW}
\newtheorem*{GM}{Theorem GM}
\newtheorem{example}{Example}

\begin{document}
\author{Ushangi Goginava }
\title[]{Limits of sequence of Tensor Product Operators associated with the
Walsh-Paley system}
\address{U. Goginava, Department of Mathematical Sciences \\
United Arab Emirates University, P.O. Box No. 15551\\
Al Ain, Abu Dhabi, UAE}
\email{zazagoginava@gmail.com; ugoginava@uaeu.ac.ae}
\address{˙ }
\email{ }
\author{ Farrukh Mukhamedov}
\address{F. Mukhamedov, Department of Mathematical Sciences \\
United Arab Emirates University, P.O. Box No. 15551,\\
Al Ain, Abu Dhabi, UAE}
\email{farrukh.m@uaeu.ac.ae}

\begin{abstract}
It is well-known that to establish the almost everywhere convergence of a sequence of operators on $L_1$-space, it is sufficient
to obtain a weak $(1,1)$-type inequality for the maximal
operator corresponding to the sequence of operators. However, in practical
applications, the establishment of the mentioned inequality for the
maximal operators is very tricky and difficult job. In the present paper, the
main aim is a novel outlook at above mentioned inequality for the tensor
product of two weighted one-dimensional Walsh-Fourier series. Namely, our
main idea is naturally to consider uniformly boundedness conditions for the
sequences which imply the weak type estimation for the maximal operator of
the tensor product. More precisely, we are going to
establish weak type of inequality for the tensor product of two dimensional
maximal operators while having the uniform boundedness of the corresponding
one dimensional operators. Moreover, the convergence of the tensor
product of operators is established at the two dimensional Walsh-Lebesgue points.
\end{abstract}

\maketitle

\bigskip \footnotetext{%
}

\section{Introduction}

The present paper deals with the investigation of one of the central topics
in the harmonic analysis that of the almost everywhere convergence of the
weighted Walsh-Fourier series associated with tensor products. Before, to
formulate the main results let us discuss some background information about
the main questions in such a theme.

\subsection{Background history of one-dimensional Fourier series}

Kolmogoroff \cite{kolmogoroff1923serie, kolmogororf1926s} solved Luzin's
problem by proving the existence of an integrable function with divergent
trigonometric Fourier series. Furthermore, Stein \cite{Stein} addressed a
related issue with Walsh systems, namely he linked the almost
everywhere~convergence of a sequence of operators to the boundedness of the
corresponding maximum operator. To be more precise, let $\{T_{m}\}$ be a
sequence of linear bounded operators from $L_{1}(\mathbb{I})$ into self
(where $\mathbb{I}=[0,1)$). Assume that each $T_{m}$ commutes with
translations of $\mathbb{I}$. Then almost everywhere convergence of $%
T_{m}(f)(x)$ for every $f\in L_{1}(\mathbb{I})$ is equivalent to the boundedness
of the maximal operator defined by
\begin{equation*}
T^{\ast }(f)(x)=\sup_{m\geq 1}\left\vert \left( T_{m}f\right) (x)\right\vert
\end{equation*}%
from $L_{1}(\mathbb{I})$ to the space $L_{1,\infty }(\mathbb{I}),$ i. e.%
\begin{equation}
\left\Vert T^{\ast }(f)\right\Vert _{L_{1,\infty }(\mathbb{I})}\leq
c\left\Vert f\right\Vert _{1}.  \label{wk-11}
\end{equation}%
Here the space $L_{1,\infty }(\mathbb{I})$ consists of all measurable
functions $f\in L_{0}\left( \mathbb{I}\right) $ such that
\begin{equation*}
\left\Vert f\right\Vert _{L_{1,\infty }(\mathbb{I})}:=\sup\limits_{t>0}t\mu
\left( \left\vert f\right\vert >t\right) <\infty ,
\end{equation*}%
where $\mu $ denotes the one-dimensional Lebesgue measure of the set.

However, from the Stein's result one yields that the inequality (\ref{wk-11}%
) does not hold for the Walsh-Carleson operator, the last one is defined by
\begin{equation*}
WCf\left( x\right) =\sup\limits_{n\geq 0}\left\vert f\ast D_{n}\right\vert ,
\end{equation*}%
where $D_{n}$ represents the Dirichlet kernel with relation to the
Walsh-Paley system. In this context, it is reasonable to explore the
weighted maximum operators that guarantee the fulfilment of inequality (\ref%
{wk-11}) and, as a result, the almost everywhere convergence of the
corresponding operator sequences would be obtained, for any integrable
function.

To be more precise, let us consider an infinite matrix $\mathbb{T}:=\left(
t_{k,n}\right) $ satisfying the following conditions \footnote{%
In the sequel, by $\mathbb{P}$ we denote the set of positive integers, and $%
\mathbb{N}\mathbf{:=}\mathbb{P}\mathbf{\cup \{}0\mathbf{\}}$. The following
notations will also be used below: $\mathbb{P}^{2}:=\mathbb{P}\times \mathbb{%
P},\mathbb{N}^{2}:=\mathbb{N}\times \mathbb{N}$.}:

\begin{itemize}
\item[(a)] $t_{k,n}\geq 0$, $k,n\in \mathbb{N};$

\item[(b)] $0\leq t_{k+1,n}\leq t_{k,n},k,n\in \mathbb{N}$;

\item[(c)] $\sum\limits_{k=0}^{n}t_{k,n}=1.$
\end{itemize}

In the sequel, such kind of matrix is called \textit{matrix of transformation%
}.

Now, we define a sequence of operators associated with $\mathbb{T}$ and the
Walsh-Fourier series as follows
\begin{equation}
\mathcal{T}^{\mathbb{T}}_{n}(f;x):=\sum_{k=0}^{n}t_{n-k,n}S_{k}(f;x)\quad
(n\in {\mathbb{N}}),  \label{weight-mean}
\end{equation}%
where $S_{k}(f;x)$ denotes the $k$-th partial sum of the Walsh-Fourier
series of a function $f$ at a point $x$.

Then, the Walsh-Carleson weighted operator is defined by
\begin{equation}
\mathcal{T}^{\ast,\mathbb{T} }\left( f\right) :=\sup\limits_{n\in \mathbb{N}%
}\left\vert \mathcal{T}^{\mathbb{T}}_{n}(f)\right\vert .  \label{WC}
\end{equation}

\begin{remark}\label{m11}
Let us consider some particular cases of the matrix transformations $\mathbb{%
T}$ and the corresponding operators $\mathcal{T}_n$.

\begin{itemize}
\item[(I)] Assume that $t_{k,n}:=1$ when $k=0$ and equals $0$ for the other $%
k$. Then $\mathcal{T}_{n}(f;x)=S_{n}(f;x)$.

\item[(II)] Let $\mathbb{F}=(t_{k,n})$ be a matrix transformation such that $%
t_{k,n}=\frac{1}{n}$, then $\mathcal{T}_{n}^{\mathbb{F}}\left( f\right)
=\sigma _{n}\left( f\right) $ (Fejér means).

\item[(III)] Now, let $\mathbb{A}=(t_{k,n})$ be a matrix transformation such
that $t_{k,n}:=\frac{A_{k}^{\alpha _{n}-1}}{A_{n}^{\alpha _{n}}}$ $\left(
\alpha _{n}\in \left( 0,1\right] \ \ n\in \mathbb{N}\right) ,$ then $%
\mathcal{T}_{n}^{\mathbb{A}}\left( f\right) =\sigma _{n}^{\alpha _{n}}\left(
f\right) $ ($\left( C,\alpha _{n}\right) $-means).

\item[(IV)] Consider $\mathbb{L}=(t_{k,n})$ a matrix transformation such
that $t_{k,n}:=\frac{1}{\ell _{n}}\frac{1}{k+1},$ where $\ell
_{n}:=\sum_{k=0}^{n}\frac{1}{k+1}$, then $\mathcal{T}_{n}^{\mathbb{L}}\left(
f\right) =M_{n}\left( f\right) $ (Nörlund logarithmic means).
\end{itemize}
\end{remark}

\begin{remark}
\label{3} It is stressed that in all above considered examples, the sequence
$t_{k,n}$ is nonincreasing with respect to $k$, while $n$ is fixed. In the
current paper, we do not investigate the scenario when $t_{k,n}$ increases ($%
n$ is fixed), since in that case, the inequality (\ref{wk-11}) is always
satisfied, and hence the sequence (\ref{weight-mean}) is almost everywhere
convergent for every integrable function.
\end{remark}

\begin{remark}
\label{4} We point out that in example (I), the maximal operator $\mathcal{T}%
^{\ast }\left( f\right) $ clearly corresponds with the Walsh-Carleson
operator, and so, as we mentioned earlier, by Stein's result reveals \cite%
{Stein} that the inequality (\ref{wk-11}) fails.

In the case (II), the maximal operator $\mathcal{T}^{\ast ,\mathbb{F}}\left(
f\right) $ enjoys the inequality (\ref{wk-11}) \cite{Fine1955cesaro}.

The case (III) may be separated into two scenarios w.r.t. $\{\alpha _{n}\}$:
(a) the sequence $\{\alpha _{n}\}$ is stationary, i.e. $\alpha _{n}=\alpha
\in \left( 0,1\right] $, and (b) when $\{\alpha _{n}\}$ approaches to zero.
In the first situation, the inequality (\ref{wk-11}) was established in \cite%
{schipp1975certain}, while in the second case, the violation of the
inequality has been recently demonstrated in \cite{GoginavaJMAA23} no matter
how slowly the sequence goes to zero.

In the example (IV), the inequality is not satisfied, for $\mathcal{T}^{\ast,%
\mathbb{L} }\left( f\right)$ (see, \cite{GatGogAMSDiv2}).
\end{remark}

From these observations, we infer that it is important to find necessary and
sufficient conditions for the fulfillment of \eqref{wk-11}. Recently, in \cite%
{GMFAA} a sufficient conditions have been established in terms of the
uniform boundedness of the operator sequences $\{\mathcal{T}_n\}$. 

Recall that by $H_{1}\left( \mathbb{I}\right)$ denote the dyadic Hardy space. Now, we formulate next result taken from \cite{GMFAA} which will be employed to prove one of our main result (see Theorem \ref{mt1}).

\begin{GM}
\label{eq} Let $\mathbb{T}$ be a matrix of transformation, and $\left\{
n_{a}:a\in \mathbb{N}\right\} $ be a subsequence of natural numbers.
Consider the sequence of $\left\{ \mathcal{T}_{n_{a}}^{\mathbb{T}}\right\} $
given by \eqref{weight-mean}. Then the following statements are equivalent:

\begin{itemize}
\item[(i)] one has
\begin{equation*}
\sup\limits_{a\in \mathbb{N}}\Vert \mathcal{T}^{\mathbb{T}}_{n_{a}}\Vert
_{H_{1}\left( \mathbb{I}\right) \rightarrow L_{1}\left( \mathbb{I}\right)
}<\infty;
\end{equation*}

\item[(ii)] one has
\begin{equation*}
\sup\limits_{a\in \mathbb{N}}\Vert \mathcal{T}^{\mathbb{T}}_{n_{a}}\Vert
_{L_{\infty }\left( \mathbb{I}\right) \rightarrow L_{\infty }\left( \mathbb{I%
}\right) }<\infty;
\end{equation*}

\item[(iii)] one has
\begin{equation*}
\sup\limits_{a\in \mathbb{N}}\upsilon (n_{a},\mathbb{T})<\infty,
\end{equation*}
where
\begin{equation*}
\upsilon (n,\mathbb{T}):=\sum_{k=0}^{|n|}|\varepsilon _{k}(n)-\varepsilon
_{k+1}(n)|\tau _{2^{k},n},
\end{equation*}
and $\tau_{s,n}:=\sum\limits_{l=0}^{s}t_{l,n}$\footnote{%
For every given number $n\in \mathbb{P}$, the following unique
representation has a place%
\begin{equation*}
n=\sum\limits_{k=0}^{\infty }\varepsilon _{k}\left( n\right) 2^{k}
\end{equation*}
where $\varepsilon _{k}\left( n\right) \in \left\{ 0,1\right\} $ and $%
\varepsilon _{k}\left( n\right) $ will be called the binary coefficients of $%
n$. The Walsh(-Paley) system $\left( w_{n}:n\in \mathbb{N}\right) $ was
introduced by Paley in 1932 as product of Rademacher functions (see \cite[p.
1]{SWS} )}.
\end{itemize}

\end{GM}

\subsection{Background history of two-dimensional Fourier series}

In what follows, by $D_{n}$ we denote the Walsh-Dirichlet kernel of order $n$
(see \cite[p.27]{SWS}). Let, as before, $\mathbb{T}:=\left( t_{k,n}:k,n\in
\mathbb{N}\right) $ be a matrix transformation. Consider the sequence of
kernels $V^{\mathbb{T}}_{n} $ resulting from \eqref{weight-mean} via
Walsh-Dirichlet kernels $D_{n}$, i.e.
\begin{equation}
{V}^{\mathbb{T}}_{n}(u):=\sum_{k=1}^{n}t_{n-k,n}D_{k}(u).  \label{V}
\end{equation}%
As we mentioned earlier, if consider the matrix ${\mathbb{F}}$, then ${V}^{%
\mathbb{F}}_{n}$ coincides with Fej\'er's kernel, which is denoted by $K_{n}$.

One can establish that
\begin{equation}
\mathcal{T}^{\mathbb{T}}_{n}f\left( x\right) =\left( f\ast {V}^{\mathbb{T}%
}_{n}\right) \left( x\right):=\int_{\mathbb{I}}f(t){V}^{\mathbb{T}%
}_{n}(x\dotplus t)d\mu(t).  \label{tn=}
\end{equation}

Now, assume that $f$ is an integrable function of two variables, i.e. $f\in
L_{1}\left( \mathbb{I}^{2}\right) $. We define the sequence of operators $\{%
\mathfrak{T}_{n^{0},n^{1}}f\left( x^{0},x^{1}\right) \}_{\left(
n^{0},n^{1}\right) \in \mathbb{N}^{2}}$ as a tensor product of the sequences
of one-dimensional operators \eqref{tn=} . More exactly,%
\begin{equation}
\mathfrak{T}_{n^{0},n^{1}}f\left( x^{0},x^{1}\right) =\left( f\ast \left(
V_{n^{0}}^{^{\mathbb{T}_{0}}}\otimes V_{n^{1}}^{^{\mathbb{T}_{1}}}\right)
\right) \left( x^{0},x^{1}\right) .  \label{tens}
\end{equation}%
Here, as before, $V_{n^{0}}^{\mathbb{T}_{0}}$ and $V_{n^{1}}^{\mathbb{T}%
_{1}} $ are the kernels associated with the matrices $\mathbb{T}_{0}$ and $%
\mathbb{T}_{1}$ of transformation, respectively.

One of the important themes of the multidimensional Fourier series is to
investigate the almost everywhere convergence of the sequence $\{\mathcal{T}%
_{n^{0},n^{1}}f\left( x^{0},x^{1}\right) \}_{\left( n^{0},n^{1}\right) \in
\mathbb{N}^{2}}$. For example, to achieve the convergence in measure or norm
of the sequence \eqref{tens}, the iteration technique is applied to convert
a two-dimensional Fourier series sequence into two one-dimensional
operators. To achieve it, we need to defined the operators $\mathcal{T}%
_{n^{0}}^{\mathbb{T}_{0}}$ and $\mathcal{T}_{n^{1}}^{\mathbb{T}_{1}}$ on $%
L_{1}\left( \mathbb{I}^{2}\right) $ by
\begin{eqnarray}
&&\mathcal{T}_{n^{0}}^{\mathbb{T}_{0}}f\left( x^{0},x^{1}\right) =\int_{%
\mathbb{I}}f(u^{0},x^{1}){V}_{n^{0}}^{\mathbb{T}_{0}}(x^{0}\dotplus
u^{0})d\mu (u^{0}).  \label{iter1} \\[2mm]
&&\mathcal{T}_{n^{1}}^{\mathbb{T}_{1}}f\left( x^{0},x^{1}\right) =\int_{%
\mathbb{I}}f(x^{0},u^{1}){V}_{n^{1}}^{\mathbb{T}_{1}}(x^{1}\dotplus
u^{1})d\mu (u^{1}).  \label{iter2}
\end{eqnarray}%
Symbolically, the defined operators can be written as follows:
\begin{eqnarray*}
&&\mathcal{T}_{n^{0}}^{\mathbb{T}_{0}}f\left( x^{0},x^{1}\right) =\left(
f\ast V_{n^{0}}^{\mathbb{T}_{0}}\right) \left( x^{0},x^{1}\right) , \\[2mm]
&&\mathcal{T}_{n^{1}}^{\mathbb{T}_{1}}f\left( x^{0},x^{1}\right) =\left(
f\ast V_{n^{1}}^{\mathbb{T}_{1}}\right) \left( x^{0},x^{1}\right) .
\end{eqnarray*}

From \eqref{iter1} one infers that the operator $\mathcal{T}_{n^{0}}^{%
\mathbb{T}_{0}}$ only acts on~the first variable and leaves the second
variable fixed. Similarly, the operator $\mathcal{T}_{n^{1}}^{\mathbb{T}%
_{1}} $ only acts on~the second variable and leaves the first one fixed.
Because the function $f$ is integrable on the square $\mathbb{I}^{2}$, the
sequence of operators \eqref{weight-mean} applies for almost all $x^{1}$. It
then follows from (\ref{iter1}), (\ref{iter2}) that
\begin{equation}
\mathfrak{T}_{n^{0},n^{1}}f\left( x^{0},x^{1}\right) =\left( \mathcal{T}%
_{n^{1}}^{\mathbb{T}_{1}}\left( \mathcal{T}_{n^{0}}^{\mathbb{T}_{0}}f\right)
\right) \left( x^{0},x^{1}\right) =\left( \mathcal{T}_{n^{0}}^{\mathbb{T}%
_{0}}\left( \mathcal{T}_{n^{1}}^{\mathbb{T}_{1}}f\right) \right) \left(
x^{0},x^{1}\right) .  \label{iter}
\end{equation}

Therefore, in what follows, the defined operator is denoted by
\begin{equation}
\mathfrak{T}_{n^{0},n^{1}}=\mathcal{T}_{n^{0}}^{\mathbb{T}_{0}}\otimes
\mathcal{T}_{n^{1}}^{\mathbb{T}_{1}}.  \label{tens1}
\end{equation}

The equation (\ref{iter}) allows to transform one-dimensional findings to
the two-dimensional setting. However, a class deformation may occur during
such a passage. On the other hand, the almost everywhere convergence is
tightly connected with the maximum operator's boundedness, unfortunately,
the same logic is not applicable in such a situation. The widely known
continuous Fefferman's function \cite{Feff-div-1971} is an example of this
phenomenon. For instance, the rectangular partial sums of its double Fourier
trigonometric series do not converge almost everywhere for that continuous
function. In this direction, Getsadze \cite{Gets-ContDiv-1985,
Gets-ConDiv-2020, Gets-ContDiv-2021} has developed an analogous example for
the Walsh system.

\begin{remark}
Let us highlight some well-known results related to the two dimensional
Fourier/Walsh-Fourier series.
\end{remark}

\begin{itemize}
\item Jessen, Marcinkiewicz and Zygmund \cite{jessen1935note,
marcinkiewicz1939summability} considered $\mathcal{T}_{n^{0}}^{\mathbb{F}%
}\otimes \mathcal{T}_{n^{1}}^{\mathbb{F}}$ with regard to the trigonometric
system. It was established that if $f\in L\ln L({\mathbb{I}}^{2})$ \footnote{%
The positive logarithm function is defined as $\ln ^{+}(x):=\mathbf{1}%
_{\left\{ x>1\right\} }\ln (x).$ We say that the function $f\in L_{1}(%
\mathbb{I}^{2})$ belongs to the logarithmic space $L\ln L({\mathbb{I}}^{2})$
if the integral $\int_{\mathbb{I}^{2}}|f|\ln ^{+}|f|$ is finite.}, then
sequence of operators $\mathcal{T}_{n^{0}}^{\mathbb{F}}\otimes \mathcal{T}%
_{n^{1}}^{\mathbb{F}}f$ converges almost everywhere to $f$.

\item Herriot \cite{Herriot1,Herriot2} investigated almost everywhere
convergence, pointwise convergence, as well as issues with the localization
of sequences of operators derived by matrix transformations for double
trigonometric Fourier series.

\item Furthermore, Moricz, Shipp, Wade \cite{MorSchWade-1992} have developed
a similar theory for the Walsh-Paley system. Later on, Gát \cite%
{GatProc-2000} showed that the mentioned results by Móricz, Schipp, and Wade
cannot be further sharpened.

\item In \cite{weisz1996cesaro, WeiszUnrestr-2000, Weisz-book2} Weisz used
dyadic Hardy spaces to investigate the almost everywhere summability of
Walsh-Fourier series. In particularly, he proved that if $f$ belongs to $%
L\ln L({\mathbb{I}}^{2})$, then the sequence of operators $\mathcal{T}%
_{n^{0}}^{\mathbb{A}}\otimes \mathcal{T}_{n^{1}}^{\mathbb{A}}f$with regard
to the Walsh-Paley system converges almost everywhere to $f$. Here the
matrix $\mathbb{A}$ is given in Remark \ref{m11} (III).

\item In this vein, we mention that Olevskii \cite{OleCont-1961,
OlevONS1961, OlevONS1963}, Kazarian \cite{KazDiv1991, KazDiv2003, KazDiv2004}%
, Karagulyan \cite{KarDiv1989, KarDiv2013}, and Getsadze \cite{Gets2007,
GetsDiv2013} have investigated the divergence of operator sequences with
respect to a general orthonormal systems.

\item Recently, Gát and Karagulyan \cite{KarTens2016} have established that $%
L\ln L({\mathbb{I}}^{2})$ space is a maximum Orlicz space where the sequence
of operators $\mathcal{T}_{n^{0}}^{\mathbb{T}_{0}}\otimes \mathcal{T}%
_{n^{1}}^{\mathbb{T}_{1}}$ converges almost everywhere to $f$ as $\min
\left\{ n^{0},n^{1}\right\} \rightarrow \infty $.

\item Oniani \cite{Oniani2023} generalized Stein and Sawyer's weak type
maximal principles to nets of operators based on functions in general
measure spaces, including locally compact groups.
\end{itemize}

Móricz, Schipp, and Wade \cite{MorSchWade-1992} proved the following
theorem, which enables iteration of the appropriate maximal operator of a
two-parameter sequence.

\begin{MSW}
\label{aa}Let $\left( V_{n}^{i},n\in \mathbb{N}\right)$, $i=0,1$ be the
sequence of $L_{1}\left( \mathbb{I}\right) $ functions. Define
one-dimensional operators%
\begin{equation*}
T^{i}f:=\sup\limits_{m\in \mathbb{N}}\left\vert f\ast V_{m}^{i}\right\vert ,%
\widetilde{T}^{i}f:=\sup\limits_{m\in \mathbb{N}}\left\vert f\ast \left\vert
V_{m}^{i}\right\vert \right\vert, \ \ f\in L_{1}\left( \mathbb{I}\right),
i\in\{0,1\}.
\end{equation*}%
Assume that there exist absolute constants $c_{0},c_{1}$ such that%
\begin{equation}
\mu \left( \left\{ \widetilde{T}^{0}f>\lambda \right\} \right) \leq \frac{%
c_{0}}{\lambda }\left\Vert f\right\Vert _{1}  \label{wk}
\end{equation}%
and%
\begin{equation}
\left\Vert T^{1}f\right\Vert _{1}\leq c_{1}\left\Vert f\right\Vert _{H_{1}}
\label{H1}
\end{equation}%
for $f\in L_{1}\left( \mathbb{I}\right) $ and $\lambda >0$. If%
\begin{equation}
Tf:=\sup\limits_{\left( n,m\right) \in \mathbb{N}^{2}}\left\vert f\ast
\left( V_{n}^{0}\otimes V_{m}^{1}\right) \right\vert ,  \label{KP}
\end{equation}%
then \footnote{We use the notation $a \lesssim b$ for
$a<c \cdot b,$ where $c>0$ is an absolute constant. }
\begin{equation}
\left\Vert Tf\right\Vert _{L_{1,\infty }(\mathbb{I}^2)}\lesssim
1+\int\limits_{\mathbb{I}^{2}}\left\vert f\right\vert \ln ^{+}\left\vert
f\right\vert  \quad \left( f\in L\ln L\left( \mathbb{I}^{2}\right)
\right) .  \label{wk-1}
\end{equation}
\end{MSW}

According to the theorem, knowing the inequalities (\ref{wk}) and (\ref{H1})
for one-dimensional maximal operators allows for the establishment of a
weak-type inequality for the maximal operator associated with the tensor
product of the provided one-dimensional operators. However, in practical
applications, the establishment of the mentioned inequalities for the
maximal operators is very tricky and difficult job. In the present work, the
main goal is a novel outlook at above mentioned inequalities for the tensor
product of two weighted one-dimensional Walsh-Fourier series. Namely, our
main idea is naturally to replace the above stated inequalities (\ref{wk})
and (\ref{H1}) with weaker uniformly boundedness conditions for the
sequences which imply the weak type estimation for the maximal operator of
the tensor product.

The following is the major move result that achieved this goal.

\begin{theorem}
\label{mt1} Let $\mathbb{T}_{i}$, $(i=0,1)$ be matrices of transformation,
and $\left\{ n_{a}^{i}:\ a\in \mathbb{N}\right\} $ be subsequences of
natural numbers. Let $\mathcal{T}_{n_{a}^{0}}^{\mathbb{T}_{0}}$ and $%
\mathcal{T}_{n_{a}^{1}}^{\mathbb{T}_{1}}$ be the corresponding sequences of
operators on $L_{1}\left( \mathbb{I}\right) $ given by \eqref{weight-mean}.
Assume that for each $i=0,1$ one of the following conditions are satisfied:

\begin{itemize}
\item[(i)] one has
\begin{equation}
\sup\limits_{a\in \mathbb{N}}\Vert \mathcal{T}_{n_{a}^{i}}^{\mathbb{T}%
_{i}}\Vert _{H_{1}\left( \mathbb{I}\right) \rightarrow L_{1}\left( \mathbb{I}%
\right) }<\infty ;  \label{h1}
\end{equation}

\item[(ii)] one has
\begin{equation*}
\sup\limits_{a\in \mathbb{N}}\Vert \mathcal{T}^{\mathbb{T}_i}_{n^i_{a}}\Vert
_{L_{\infty }\left( \mathbb{I}\right) \rightarrow L_{\infty }\left( \mathbb{I%
}\right) }<\infty;
\end{equation*}

\item[(iii)] one has
\begin{equation*}
\sup\limits_{a\in \mathbb{N}}\upsilon (n^i_{a},\mathbb{T}_i)<\infty.
\end{equation*}
\end{itemize}

If%
\begin{equation*}
\mathfrak{T}^{\mathbb{T}_{0},\mathbb{T}_{1}}f:=\sup\limits_{(a,b)\in \mathbb{%
N}^{2}}\left\vert (\mathcal{T}_{n_{a}^{0}}^{\mathbb{T}_{0}}\otimes \mathcal{T%
}_{n_{b}^{1}}^{\mathbb{T}_{1}})f\right\vert
\end{equation*}%
then%
\begin{equation}
\left\Vert \mathfrak{T^{\mathbb{T}_{0},\mathbb{T}_{1}}}f\right\Vert
_{L_{1,\infty }(\mathbb{I}^{2})}\lesssim 1+\int\limits_{\mathbb{I}%
^{2}}\left\vert f\right\vert \ln ^{+}\left\vert f\right\vert \quad \left(
f\in L\ln L\left( \mathbb{I}^{2}\right) \right) .  \label{llogl}
\end{equation}
\end{theorem}

\begin{remark}
Let $f\in L_{1}({\mathbb{I}}^{2})$, the hybrid maximal function is
introduced by
\begin{equation*}
f^{\natural }(x,y):=\sup_{n\in {\mathbb{N}}}\frac{1}{|I_{n}(x)|}\left\vert
\int_{I_{n}(x)}f(t,y)d\mu \left( t\right) \right\vert .
\end{equation*}%
Define the Hardy hybrid space by 
\begin{equation*}
H_{\natural }({\mathbb{I}}^{2})=\left\{f\in L_{1}({\mathbb{I}}^{2}): \ 
 \Vert f\Vert _{H_{\natural}}:=\left\Vert f^{\natural }\right\Vert _{L_1(\mathbb{I}^2)}<\infty\right\} .
\end{equation*}%
It is well-known that  $L\ln L({\mathbb{I}}^{2})\subset H_{\natural }({\mathbb{I}}^{2})$%
. Moreover, $f\in L\ln L({\mathbb{I}}^{2})$ if and only if $|f|\in
H_{\natural }^{1}({\mathbb{I}}^{2})$. We are concerned if Theorem \ref{mt1} will hold true if the inequality (\ref{llogl}) is replaced out with a weaker inequality. Namely, in the right side of the inequality (\ref{llogl}) the norm $\left\Vert \cdot \right\Vert _{L\log L\left( \mathbb{I}^{2}\right) }$  is replaced by the norm $\left\Vert \cdot \right\Vert _{H^{\#}\left( \mathbb{I}^{2}\right) }$. However, if we screen  the proof of Theorem \ref{mt1}, then the validity of the following inequality
\begin{equation*}
\left\Vert \mathfrak{T^{\mathbb{T}_{0},\mathbb{T}_{1}}}f\right\Vert
_{L_{1,\infty }(\mathbb{I}^{2})}\lesssim \left\Vert f\right\Vert
_{H_{\natural }({\mathbb{I}}^{2})},\quad  f\in H_{\natural }({\mathbb{I}%
}^{2})
\end{equation*}%
can be proven if the condition (\ref{h1}) holds for one sequence, say $\mathcal{T%
}_{n_{a}^{1}}^{\mathbb{T}_{1}}$, and the second one satisfies the
following maximal inequality%
\begin{equation*}
\left\Vert \sup\limits_{a\in \mathbb{N}}\left\vert \mathcal{T}_{n_{a}^{2}}^{%
\mathbb{T}_{2}}\right\vert \right\Vert _{H_{1}\left( \mathbb{I}\right)
\rightarrow L_{1}\left( \mathbb{I}\right) }<\infty .
\end{equation*}
\end{remark}

\subsection{Lebesgue points and their application in the pointwise
convergence of sequences of operators}

It is usually assumed that convergence almost always implies convergence,
with the exception of a set of measure 0. However, because a set of measure
zero might be complex, convergence at one point cannot be deduced from
convergence practically everywhere. As a result, it is useful to to
establish a sufficient condition at a specific point to check the sequence's
convergence, and then count the number of such locations. In this regards,
Lebesgue explored the arithmetic means of partial sums of trigonometric
Fourier series and established the idea of a point that bears his name.

Given $f\in L_{1}\left( \mathbb{I}\right) $, a point $x\in \mathbb{I}$ is
\textit{a Lebesgue point} of $f$ if \cite[pp 41-42]{TorchBook}
\begin{equation}
\lim_{\varepsilon \rightarrow 0+}\frac{1}{\varepsilon }\int\limits_{\left[
0,\varepsilon \right] }\left\vert f\left( x+t\right) -f\left( x\right)
\right\vert dt=0.  \label{LP}
\end{equation}%
It is known that for every integrable function almost every point is a
Lebesgue point. Furthermore, Lebesgue's theorem states that the arithmetic
means of partial sums of trigonometric Fourier series is convergent at
Lebesgue points. For the Walsh system, introducing Lebesgue points in the
form (\ref{LP}) is inaccurate because the condition (\ref{LP}) does not
guarantee the convergence of arithmetic means of partial sums at a point. To
be more precise, let us illustrate it in the next example.

\begin{example}
\label{ex1} Take a subsequence of natural numbers $\left\{ n_{k}:k\in
\mathbb{N}\right\} $ enjoying the following conditions:%
\begin{equation}
n_{k}>3n_{k-1}  \label{n1}
\end{equation}%
and%
\begin{equation}
n_{k}>2^{2k}.  \label{n2}
\end{equation}%
Take a sequence of functions $\left\{ f_{k}:k\in \mathbb{N}\right\} $
defined by
\begin{equation}
f_{k}\left( x\right) :=\sum\limits_{a=n_{k-1}+1}^{n_{k}}2^{n_{k}-a}\mathbf{1}%
_{\left[ 2^{-a},2^{-a}+2^{-n_{k}}\right) },  \label{fk}
\end{equation}%
where $\mathbf{1}_{E}$ is the characteristic function of the set $E$.

Set%
\begin{equation*}
f\left( x\right) :=\sum\limits_{k=1}^{\infty }\frac{f_{k}\left( x\right) }{%
2^{k}},\ \ f\left( 0\right) =0.
\end{equation*}%
It is easy to see that $f\in L_{1}\left( \mathbb{I}\right) $. Now, let us
show that the point $x=0$ is a Lebesgue point. Indeed, for every $%
\varepsilon >0$, there exists $k\in \mathbb{N}$ such that $2^{n_{k-1}}\leq
\frac{1}{\varepsilon }<2^{n_{k}}$. Then, one has
\begin{eqnarray}
\frac{1}{\varepsilon }\int\limits_{\left[ 0,\varepsilon \right] }\left\vert
f\left( t\right) -f\left( 0\right) \right\vert dt &=&\frac{1}{\varepsilon }%
\int\limits_{\left[ 0,\varepsilon \right] }f\left( t\right) d\mu \left(
t\right)   \label{LP2} \\
&=&\frac{1}{\varepsilon }\sum\limits_{s=k}^{\infty }\frac{1}{2^{s}}%
\int\limits_{\left[ 0,\varepsilon \right] }f_{s}\left( t\right) d\mu \left(
t\right)   \notag \\
&=&\frac{1}{\varepsilon }\frac{1}{2^{k}}\int\limits_{\left[ 0,\varepsilon %
\right] }f_{k}\left( t\right) d\mu \left( t\right) +\frac{1}{\varepsilon }%
\sum\limits_{s=k+1}^{\infty }\frac{1}{2^{s}}\int\limits_{\left[
0,\varepsilon \right] }f_{s}\left( t\right) d\mu \left( t\right) .  \notag
\end{eqnarray}%
Assume that $s>k$. Then, by \eqref{fk} we have
\begin{equation}
\int\limits_{\left[ 0,\varepsilon \right] }f_{s}\left( t\right) d\mu \left(
t\right) =\frac{1}{2^{n_{s}}}+\cdots +\frac{1}{2^{n_{s-1}+1}}<\frac{1}{%
2^{n_{k}}}.  \label{LP3}
\end{equation}%
For small $\varepsilon >0$ there exists $l\in \left\{ 0,1,\dots
,n_{k}-n_{k-1}\right\} $ such that
\begin{equation*}
2^{n_{k-1}+l}\leq \frac{1}{\varepsilon }<2^{n_{k-1}+l+1}.
\end{equation*}%
Hence,
\begin{eqnarray}
\frac{1}{\varepsilon }\int\limits_{\left[ 0,\varepsilon \right] }f_{k}\left(
t\right) d\mu \left( t\right)  &\leq &2^{n_{k-1}+l+1}\int\limits_{\left[
0,2^{-n_{k-1}-l}\right] }f_{k}\left( t\right) d\mu \left( t\right)
\label{LP4} \\
&\leq
&2^{n_{k-1}+l+1}\sum\limits_{a=n_{k-1}+l}^{n_{k}}2^{n_{k}-a}2^{-n_{k}}\leq
c<\infty   \notag
\end{eqnarray}%
Now, plugging (\ref{LP3}) and (\ref{LP4}) into (\ref{LP2}) yields
\begin{equation*}
\frac{1}{\varepsilon }\int\limits_{\left[ 0,\varepsilon \right] }\left\vert
f\left( t\right) -f\left( 0\right) \right\vert d\mu \left( t\right) \leq
\frac{c}{2^{k}}\rightarrow 0\text{ as }\varepsilon \rightarrow 0,
\end{equation*}%
which means $x=0$ is a Lebesgue point for $f$.

Now, let us demonstrate the divergence of $\sigma _{2^{n_{k}}}\left(
f,0\right) $. Indeed, by \cite{KK, p.46}, one has
\begin{equation*}
K_{2^{n}}\left( x\right) =\frac{1}{2}\left( \left( 2^{n}+1\right) \mathbf{1}%
_{I_{n}\left( x\right) }+\sum\limits_{j=0}^{n-1}2^{j}\mathbf{1}_{I_{n}\left(
x\dotplus 2^{-j-1}\right) }\right) .
\end{equation*}%
Due to (\ref{n1}) and (\ref{n2}) we obtain
\begin{eqnarray*}
\sigma _{2^{n_{k}}}\left( f,0\right)  &\geq &\frac{1}{2^{k}}%
\sum\limits_{j=n_{k-1}}^{n_{k}-1}2^{j-1}\int\limits_{I_{n_{k}}\left(
2^{-j-1}\right) }f_{k}\left( t\right) d\mu \left( t\right)  \\
&=&\frac{1}{2^{k}}\sum%
\limits_{j=n_{k-1}}^{n_{k}-1}2^{j-1}2^{n_{k}-j}2^{-n_{k}} \\
&=&\frac{n_{k}-n_{k-1}}{2^{k+1}}\geq c2^{k}\rightarrow \infty \text{ as }%
k\rightarrow \infty .
\end{eqnarray*}
\end{example}

\begin{remark}
Let us consider the Fejér matrix of transformation $\mathbb{F}$. In such a
situation, for the corresponding kernel one has
\begin{equation}
\sup\limits_{n\in \mathbb{N}}\left\Vert K_{n}\right\Vert _{1}<\infty
\label{kn}
\end{equation}%
for both the trigonometric and Walsh systems. We point out that, in the
trigonometric system, the function $\left\vert K_{n}\right\vert $ may be
estimated by an integrable, decreasing, non-negative function which yields
that the operator $f\ast K_{n}$ can be evaluated from above by the maximum
Hardy-Littlewood function (see \cite[pp. 41-43]{TorchBook}). However,
despite the inequality (\ref{kn}), the Walsh system lacks a non-negative,
decreasing, and integrable majorant for $\left\vert K_{n}\right\vert $. As a
result, the above-mentioned Lebesgue point does not apply in such a
situation.
\end{remark}

Given the circumstances described above, it seemed obvious to define an
analog of Walsh-Lebesgue points. In this direction, Weisz \cite{WeiszWLP1}
first introduced the idea of Walsh-Lebesgue points. Namely, an element $x\in
\mathbb{I}$ is a Walsh-Lebesgue point of an integrable function $f\in
L_{1}\left( \mathbb{I}\right) $ if%
\begin{equation*}
\lim_{n\rightarrow \infty }W_{n}f\left( x\right) =0,
\end{equation*}%
where%
\begin{equation*}
W_{n}f\left( x\right) =\sum\limits_{k=0}^{n}2^{k}\int\limits_{I_{n}\left(
x\dotplus 2^{-k-1}\right) }\left\vert f\left( t\right) -f\left( x\right)
\right\vert d\mu \left( t\right) ,
\end{equation*}%
here, as before, $\dotplus $ stands for the dyadic sum and $I_{n}\left(
x\right) $ denotes the dyadic interval of the $n$th order, which contains
the point $x$.

Furthermore, in the mentioned paper \cite{ WeiszWLP1} it has been
established that Walsh-Fej\'er means converge at all Walsh-Lebesgue points,
which is a generalisation of Fine's \cite{Fine1955cesaro} result.

The works \cite{GogGogLebPoints,WeiszWLP3} established the concept of
Welsh-Lebesgue points for two-variable functions. Namely, given $f\in
L_{1}\left( \mathbb{I}^{2}\right) $, let us denote

\small

\begin{eqnarray}
&&W_{n^{0},n^{1}}f\left( x^{0},x^{1}\right)
:=\sum\limits_{i^{0}=0}^{n^{0}}\sum\limits_{i^{1}=0}^{n^{1}}2^{i^{0}+i^{1}}%
\int\limits_{I_{n^{0}}\left( x^{0}\dotplus 2^{-i^{0}-1}\right) \times
I_{n^{1}}\left( x^{1}\dotplus 2^{-i^{1}-1}\right) }\Delta(f)(s^{0},s^{1},x^{0},x^{1})d\boldsymbol{\mu }%
\left( s^{0},s^{1}\right) ,  \label{Wnm} \\[2mm]
&&H_{n^{1}}^{\left( 1\right) }f\left( x^{0},x^{1}\right)
:=\sum\limits_{i^{1}=0}^{n^{1}}2^{i^{1}}\int\limits_{\mathbb{I}\times
I_{n^{1}}\left( x^{1}\dotplus 2^{-i^{1}-1}\right) }\Delta(f)(s^{0},s^{1},x^{0},x^{1}) d\boldsymbol{\mu }%
\left( s^{0},s^{1}\right) ,  \label{H2} \\[2mm]
&&H_{n^{0}}^{\left( 0\right) }f\left( x^{0},x^{1}\right)
:=\sum\limits_{i^{0}=0}^{n^{0}}2^{i^{0}}\int\limits_{I_{n^{0}}\left(
x\dotplus 2^{-i^{0}-1}\right) \times \mathbb{I}}\Delta(f)(s^{0},s^{1},x^{0},x^{1}) d\boldsymbol{\mu }%
\left( s^{0},s^{1}\right) .  \label{H1}
\end{eqnarray}%
\normalsize 
where
$$
\Delta(f)(s^{0},s^{1},x^{0},x^{1})=|f(s^{0},s^{1})-f(x^{0},x^{1})|
$$
and $\boldsymbol{\mu }$ stands for the two-dimensional Lebesgue measure defined by $\boldsymbol{\mu }:=\mu\otimes \mu$.

In \cite{GogGogLebPoints} the importance of all three
conditions has been stressed by providing an example of an integrable
function for which one of the above conditions fails in a single point, then
the corresponding double Walsh-Fejer means will not converge at that point.

\begin{definition} Let $f\in L_{1}\left( \mathbb{I}^{2}\right) $. A point $%
\left( x^{0},x^{1}\right) \in \mathbb{I}^{2}\,$ is a two-dimensional
Walsh-Lebesgue point of $f$, if \begin{equation} \lim\limits_{\min \left\{
n^{0},n^{1}\right\} \rightarrow \infty }W_{n^{0},n^{1}}f\left(
x^{0},x^{1}\right) =0, \label{wl1} \end{equation}\begin{equation}
\sup\limits_{n}H_{n}^{\left( 1\right) }f\left( x^{0},x^{1}\right) <\infty
\,\,\,\,\, \label{wl2} \end{equation}and \begin{equation} \sup%
\limits_{m}H_{m}^{\left( 2\right) }f\left( x^{0},x^{1}\right) <\infty \,.
\label{wl3} \end{equation} \end{definition} The results of \cite{%
GogGogLebPoints} yield that for every $f\in L\ln L\left( \mathbb{I}%
^{2}\right) $, almost all points from its domain form two-dimensional
Walsh-Lebesgue points. According to Theorem \ref{mt1}, the sequence $(%
\mathcal{T}_{n_{a}^{0}}^{\mathbb{T}_{0}}\otimes \mathcal{T}_{n_{b}^{1}}^{%
\mathbb{T}_{1}})f$ almost everywhere convergence for $f\in L\ln L\left(
\mathbb{I}^{2}\right) $. The second main result of this paper reveals that
convergent set of that sequence contains two dimensional Walsh-Lebesgue
points of $f$. Here is formulation of the second main result.

\begin{theorem}
\label{mt2} Let us assume that the statements of Theorem \ref{mt1} are
fulfilled. Moreover, for the sequences $\left\{ n_{a}^{0}:a\in \mathbb{P}%
\right\} $ and $\left\{ n_{b}^{1}:b\in \mathbb{P}\right\} $ one has $%
t_{0,n_{a}^{i}}^{i}=o\left( 1\right) $ as $n_{a}^{i}\rightarrow \infty
\left( i=0,1\right) $. Then for each $f\in L\ln L\left( \mathbb{I}%
^{2}\right) $ the sequence $\{(\mathcal{T}_{n_{a}^{0}}^{\mathbb{T}%
_{0}}\otimes \mathcal{T}_{n_{b}^{1}}^{\mathbb{T}_{1}})f\}$ converges to $f$
at every two-dimensional Walsh-Lebesgue point as $\min \left\{
n_{a}^{0},n_{b}^{1}\right\} \rightarrow \infty $. \end{theorem}

\section{Proof of Theorem \protect\ref{mt1}}

 The present section is devoted to the
proof of Theorem \ref{mt1}. Before to establish it, we are going to prove an
auxiliary fact which has independent interest. Let, as before, $\mathbb{T}%
:=\left( t_{k,n}:k,n\in \mathbb{N}\right) $ be a matrix of transformation.
Consider the sequence of kernels $V_{n}^{\mathbb{T}}$ given by \eqref{V}.
Let us define \begin{equation} \widetilde{\mathcal{T}}_{n}^{\mathbb{T}%
}f:=f\ast \left\vert V_{n}^{\mathbb{T}}\right\vert \ \ f\in L_{1}\left(
\mathbb{I}\right) . \label{Tn}
\end{equation}

\begin{prop} \label{MMo} Let $%
\mathbb{T}$ be a matrix transformation, and $\left\{ n_{b}:b\in \mathbb{N}%
\right\} $ be a subsequence of natural numbers. Assume that one of the
conditions (i)-(iii) of Theorem GM holds. Then \begin{equation} \left\Vert
\sup\limits_{b\in \mathbb{N}}\left\vert \widetilde{\mathcal{T}}_{n_{b}}^{%
\mathbb{T}}f\right\vert \right\Vert _{L_{1,\infty }(\mathbb{I})}\lesssim
\left\Vert f\right\Vert _{1}\text{ \ }\left( f\in L_{1}\left( \mathbb{I}%
\right) \right) \label{wk11} \end{equation}and\begin{equation} \left\Vert
\sup\limits_{b\in \mathbb{N}}\left\vert \widetilde{\mathcal{T}}_{n_{b}}^{%
\mathbb{T}}f\right\vert \right\Vert _{\infty }\lesssim \left\Vert
f\right\Vert _{\infty }\text{ \ }\left( f\in L_{\infty }\left( \mathbb{I}%
\right) \right) . \label{inf} \end{equation}
\end{prop}

\begin{proof} By
\cite{GoginavaQM2022} the function $V_{n}$ has the following representation: \begin{%
equation} V_{n_{b}}^\mathbb{T}=V_{n_{b},1}^\mathbb{T}+V_{n_{b},2}^\mathbb{T}%
, \label{V1+V2} \end{equation}
where
\begin{equation} V_{n_{b},1}^\mathbb{T}%
:=w_{n_{b}}\sum\limits_{s=0}^{|n_{b}|}\varepsilon _{s}\left( n_{b}\right)
\tau _{n_{b\left( s\right) },n_{b}}w_{2^{s}}D_{2^{s}}, \label{V11} \end{%
equation}\begin{equation*} n_{b}\left( s\right)
:=\sum\limits_{j=0}^{s}\varepsilon _{j}\left( n_{b}\right) 2^{j} 
\end{equation*}
and\begin{eqnarray} V_{n_{b},2}^\mathbb{T} &:&=-w_{n_{b}}\sum%
\limits_{s=0}^{|n_{b}|}\varepsilon _{s}\left( n_{b}\right) w_{n_{b}\left(
s\right) }w_{2^{s}-1} \label{V12} \\ &&\times
\bigg(\sum\limits_{l=1}^{2^{s}-2}\left( t_{n_{b}\left( s-1\right)
+l,n_{b}}-t_{n_{b}\left( s-1\right) +l+1,n_{b}}\right) \notag \\ && \times
lK_{l}+t_{n_{b}\left( s-1\right) +2^{s}-1,n_{b}}\left( 2^{s}-1\right)
K_{2^{s}-1}\bigg). \notag \end{eqnarray}
It then follows from (\ref{Tn}) and (%
\ref{V1+V2}) that \begin{equation} \widetilde{\mathcal{T}}^{\mathbb{T}%
}_{n_{b}}\left\vert f\right\vert \leq \left\vert f\right\vert \ast
\left\vert V_{n_{b},1}^\mathbb{T}\right\vert +\left\vert f\right\vert \ast
\left\vert V_{n_{b},2}^\mathbb{T}\right\vert . \label{V|f|} \end{equation}
Due to the equality $w_{2^{s}}D_{2^{s}}=D_{2^{s+1}}-D_{2^{s}}$ \cite[p. 34]{%
SWS}, we rewrite $V_{n_{b},1}$ as follows \begin{eqnarray} V_{n_{b},1}^%
\mathbb{T} &=&\sum\limits_{s=1}^{\left\vert n_{b}\right\vert }\left(
\varepsilon _{s-1}\left( n_{b}\right) -\varepsilon _{s}\left( n_{b}\right)
\right) \tau _{n_{b}\left( s-1\right) ,n_{b}}D_{2^{s}} \label{vv} \\
&&+\sum\limits_{s=1}^{\left\vert n_{b}\right\vert }\left( \tau _{n_{b}\left(
s-1\right) ,n_{b}}-\tau _{n_{b}\left( s\right) ,n_{b}}\right) \varepsilon
_{s}\left( n_{b}\right) D_{2^{s}} \notag \\ &&+\varepsilon _{\left\vert
n_{b}\right\vert }\left( n_{b}\right) \tau _{n_{b}\left( \left\vert
n_{b}\right\vert \right) ,n_{b}}D_{2^{\left\vert n_{b}\right\vert
+1}}-\varepsilon _{0}\left( \left\vert n_{b}\right\vert \right) \tau
_{n_{b}\left( 0\right) }, \notag \end{eqnarray}and consequently,\begin{%
eqnarray} \left\vert f\right\vert \ast \left\vert V_{n_{b},1}^\mathbb{T}%
\right\vert &\leq &\sum\limits_{s=1}^{\left\vert n_{b}\right\vert
}\left\vert \varepsilon _{s-1}\left( n_{b}\right) -\varepsilon _{s}\left(
n_{b}\right) \right\vert \tau _{2^{s},n_{b}}\left( \left\vert f\right\vert
\ast D_{2^{s}}\right) \label{v1} \\ &&+\sum\limits_{s=1}^{\left\vert
n_{b}\right\vert }2^{s}t_{2^{s},n_{b}}\left( \left\vert f\right\vert \ast
D_{2^{s}}\right) \notag \\ &&+\tau _{n_{b}\left( \left\vert n_{b}\right\vert
\right) ,n_{b}}\left( \left\vert f\right\vert \ast D_{2^{\left\vert
n_{b}\right\vert +1}}\right) +\tau_{n_{b}\left( 0\right) }\left\Vert
f\right\Vert _{1} \notag \\ &\lesssim &\upsilon (n_{a},\mathbb{T})E^{\ast
}\left( \left\vert f\right\vert \right) \lesssim E^{\ast }\left( \left\vert
f\right\vert \right) , \notag \end{eqnarray}where $E^{\ast }\left( f\right)
:=\sup\limits_{n\in \mathbb{N}}\left\vert S_{2^{n}}\left( f\right)
\right\vert $. According to \cite{GoginavaMatrix}, one has \begin{equation}
\sup\limits_{a\in \mathbb{N}}\left( \left\vert f\right\vert \ast \left\vert
V_{n_{b},2}^{\mathbb{T}}\right\vert \right) \lesssim \sup\limits_{k\in
\mathbb{N}}\left( \left\vert f\right\vert \ast \left\vert K_{k}\right\vert
\right) +E^{\ast }\left( \left\vert f\right\vert ;x\right) . \label{v2} \end{%
equation}Now, combining (\ref{Tn}), (\ref{V|f|}), (\ref{v1}) and (\ref{v2})
altogether imply \begin{equation} \sup\limits_{b\in \mathbb{N}}\left\vert
\widetilde{\mathcal{T}}_{n_{b}}^{\mathbb{T}}f\right\vert \lesssim
\sup\limits_{k\in \mathbb{N}}\left( \left\vert f\right\vert \ast \left\vert
K_{k}\right\vert \right) +E^{\ast }\left( \left\vert f\right\vert ;x\right)
. \label{add}
\end{equation}
According to \cite[Ch. 3]{SWS} and \cite{%
GogPositivity} we have \begin{equation} \left\Vert E^{\ast }\left(
\left\vert f\right\vert \right) \right\Vert _{L_{1,\infty }(\mathbb{I}%
)}\lesssim \left\Vert f\right\Vert _{1}\text{ \ }\left( f\in L_{1}\left(
\mathbb{I}\right) \right) , \label{11} \end{equation}\begin{equation}
\left\Vert \sup\limits_{k\in \mathbb{N}}\left( \left\vert f\right\vert \ast
\left\vert K_{k}\right\vert \right) \right\Vert _{L_{1,\infty }(\mathbb{I}%
)}\lesssim \left\Vert f\right\Vert _{1}\text{ \ }\left( f\in L_{1}\left(
\mathbb{I}\right) \right) . \label{11-2} \end{equation}The last inequalities
together with \eqref{add} yield the required assertions. This completes the
proof.
\end{proof}

\begin{proof}[Proof of Theorem \protect\ref{mt1}] Now,
using \eqref{tens1} and iteration together with \eqref{iter1}, one finds
\small
\begin{eqnarray} \sup\limits_{\left( a,b\right) \in \mathbb{N}%
^{2}}\left\vert f\ast \left( V_{n_{a}^{0}}^{\mathbb{T}_{0}}\otimes
V_{n_{b}^{1}}^{\mathbb{T}_{1}}\right) (x^{0},x^{1})\right\vert
&=&\sup\limits_{\left( a,b\right) \in \mathbb{N}^{2}}\left\vert \int\limits_{%
\mathbb{I}^{2}}f\left( u^{0},u^{1}\right) V_{n_{a}^{0}}^{\mathbb{T}%
_{0}}\left( x^{0}\dotplus u^{0}\right) V_{n_{b}^{1}}^{\mathbb{T}_{1}}\left(
x^{1}\dotplus u^{1}\right) d\mathbf{\mu }\left( u^{0},u^{1}\right)
\right\vert \notag \\ &\leq &\sup\limits_{a\in \mathbb{N}}\int\limits_{%
\mathbb{I}}\sup\limits_{b\in \mathbb{N}}\left\vert \int\limits_{\mathbb{I}%
}f\left( u^{0},u^{1}\right) V_{n_{b}^{1}}^{\mathbb{T}_{1}}\left(
x^{1}\dotplus u^{1}\right) du^{1}\right\vert \left\vert V_{n_{a}^{0}}^{%
\mathbb{T}_{0}}\left( x^{0}\dotplus u^{0}\right) \right\vert du^{0} \notag
\\ &=&\sup\limits_{a\in \mathbb{N}}\int\limits_{\mathbb{I}}\sup\limits_{b\in
\mathbb{N}}\left\vert \mathcal{T}_{n_{b}}^{\mathbb{T}_{1}}f\left(
u^{0},x^{1}\right) \right\vert \left\vert V_{n_{a}^{0}}^{\mathbb{T}%
_{0}}\left( x^{0}\dotplus u^{0}\right) \right\vert du^{0} \notag \\
&=&\sup\limits_{a\in \mathbb{N}}\widetilde{\mathcal{T}}_{n_{a}^{0}}^{\mathbb{%
T}_{0}}\left( \sup\limits_{b\in \mathbb{N}}\left\vert \mathcal{T}%
_{n_{b}^{1}}^{\mathbb{T}_{1}}f\right\vert \right) (x^{0},x^{1}). \label{GG}
\end{eqnarray}
\normalsize
On the other hand, by \cite{GoginavaQM2022} one gets \begin{equation*}
\left\Vert \sup\limits_{b\in \mathbb{N}}\left\vert \mathcal{T}_{n_{b}^{1}}^{%
\mathbb{T}_{1}}f\left( x^{0},\cdot \right) \right\vert \right\Vert
_{L_{1,\infty }(\mathbb{I})}\lesssim \left\Vert f\left( x^{0},\cdot \right)
\right\Vert _{1},\ \ f(x^{0},\cdot )\in L_{1}\left( \mathbb{I}\right) , \end{%
equation*}and\begin{equation*} \left\Vert \sup\limits_{b\in \mathbb{N}%
}\left\vert \mathcal{T}_{n_{b}^{1}}^{\mathbb{T}_{1}}f\left( x^{0},\cdot
\right) \right\vert \right\Vert _{\infty }\lesssim \left\Vert f\left(
x^{0},\cdot \right) \right\Vert _{\infty },\ \ f(x^{0},\cdot )\in L_{\infty
}\left( \mathbb{I}\right) \end{equation*}for a. e. $x^{0}\in \mathbb{I}$.
Therefore, using a well-known fact \cite[p. 92]{TorchBook}, the following
inequality is achieved \begin{equation} \left\Vert \sup\limits_{b\in \mathbb{%
N}}\left\vert \mathcal{T}_{n_{b}^{1}}^{\mathbb{T}_{1}}f\left( x^{0},\cdot
\right) \right\vert \right\Vert _{1}\lesssim 1+\int\limits_{\mathbb{I}%
}\left\vert f\left( x^{0},x^{1}\right) \right\vert \ln ^{+}\left\vert
f\left( x^{0},x^{1}\right) \right\vert dx^{1} \label{zy} \end{equation}where
$f\left( x^{0},\cdot \right) \in L\ln L\left( \mathbb{I}\right) $ for a. e. $%
x^{0}\in \mathbb{I}$. Hence, we can write
\begin{eqnarray*}
&&\displaystyle\left\Vert \sup\limits_{\left( a,b\right) \in \mathbb{N}%
^{2}}\left\vert f\ast \left( V_{n_{a}^{0}}^{\mathbb{T}_{0}}\otimes
V_{n_{b}^{1}}^{\mathbb{T}_{1}}\right) \right\vert \right\Vert _{L_{1,\infty
}(\mathbb{I}^{2})} \\
&\overset{\text{by (\ref{GG})}}{\lesssim }&\left\Vert
\sup\limits_{a\in \mathbb{N}}\widetilde{\mathcal{T}}_{n_{a}^{0}}^{\mathbb{T}%
_{0}}\left( \sup\limits_{b\in \mathbb{N}}\left\vert \mathcal{T}_{n_{b}^{1}}^{%
\mathbb{T}_{1}}f\right\vert \right) \right\Vert _{L_{1,\infty }(\mathbb{I}%
^{2})} \\ &\overset{\text{by definition of norm }}{=}&\sup\limits_{\lambda
>0}\lambda \left\Vert \mathbf{1}_{\left\{ \sup\limits_{a\in \mathbb{N}}%
\widetilde{\mathcal{T}}_{n_{a}^{0}}^{\mathbb{T}_{0}}\left( \sup\limits_{b\in
\mathbb{N}}\left\vert \mathcal{T}_{n_{b}^{1}}^{\mathbb{T}_{1}}f\right\vert
\right) >\lambda \right\} }\right\Vert _{L_{1}(\mathbb{I}^{2})} \\ &\overset{%
\text{by Fubin's theorem}}{=}&\sup\limits_{\lambda >0}\lambda \int\limits_{%
\mathbb{I}}\left( \int\limits_{\mathbb{I}}\mathbf{1}_{\left\{
\sup\limits_{a\in \mathbb{N}}\widetilde{\mathcal{T}}_{n_{a}^{0}}^{\mathbb{T}%
_{0}}\left( \sup\limits_{b\in \mathbb{N}}\left\vert \mathcal{T}_{n_{b}^{1}}^{%
\mathbb{T}_{1}}f\right\vert \right) >\lambda \right\} }(x^{0},x^{1})d\mu
\left( x^{0}\right) \right) d\mu \left( x^{1}\right) \\ &\leq &\int\limits_{%
\mathbb{I}}\left( \sup\limits_{\lambda >0}\lambda \int\limits_{\mathbb{I}}%
\mathbf{1}_{\left\{ \sup\limits_{a\in \mathbb{N}}\widetilde{\mathcal{T}}%
_{n_{a}^{0}}^{\mathbb{T}_{0}}\left( \sup\limits_{b\in \mathbb{N}}\left\vert
\mathcal{T}_{n_{b}^{1}}^{\mathbb{T}_{1}}f\right\vert \right) >\lambda
\right\} }(x^{0},x^{1})d\mu \left( x^{0}\right) \right) d\mu \left(
x^{1}\right) \\ &=&\int\limits_{\mathbb{I}}\sup\limits_{\lambda >0}\lambda
\mu \left\{ x^{0}:\sup\limits_{a\in \mathbb{N}}\widetilde{\mathcal{T}}%
_{n_{a}^{0}}^{\mathbb{T}_{1}}\left( \sup\limits_{b\in \mathbb{N}}\left\vert
\mathcal{T}_{n_{b}^{1}}^{\mathbb{T}_{1}}f\left( x^{0},x^{1}\right)
\right\vert \right) >\lambda \right\} d\mu \left( x^{1}\right) \\
&=&\int\limits_{\mathbb{I}}\bigg\|\sup\limits_{a\in \mathbb{N}}\widetilde{%
\mathcal{T}}_{n_{a}^{0}}^{\mathbb{T}_{1}}\left( \sup\limits_{b\in \mathbb{N}%
}\left\vert \mathcal{T}_{n_{b}^{1}}^{\mathbb{T}_{1}}f\right\vert \right)
(\cdot ,x^{1})\bigg\|_{L_{1,\infty }(\mathbb{I})}d\mu \left( x^{1}\right) \\
&\overset{\text{by Proposition \ref{MMo} }}{\leq }&\int\limits_{\mathbb{I}%
}\left\Vert \sup\limits_{b\in \mathbb{N}}\left\vert \mathcal{T}_{n_{b}^{1}}^{%
\mathbb{T}_{1}}f(\cdot ,x^{1})\right\vert \right\Vert _{1}d\mu \left(
x^{1}\right) \\ &\overset{\text{by Fubin's theorem }}{=}&\int\limits_{%
\mathbb{I}}\left\Vert \sup\limits_{b\in \mathbb{N}}\left\vert \mathcal{T}%
_{n_{b}^{1}}^{\mathbb{T}_{1}}f(x^{0},\cdot )\right\vert \right\Vert _{1}d\mu
\left( x^{0}\right) \\ &\overset{\text{by (\ref{zy}) }}{\lesssim }%
&\int\limits_{\mathbb{I}}\left( \int\limits_{\mathbb{I}}\left( 1+\left\vert
f\left( x,y\right) \right\vert \ln ^{+}\left\vert f\left( x,y\right)
\right\vert \right) d\mu \left( x^{1}\right) \right) d\mu \left(
x^{0}\right) \\ &\overset{\text{by Fubin's theorem }}{\lesssim }%
&1+\int\limits_{\mathbb{I}}\left\vert f\left( x^{0},x^{1}\right) \right\vert
\ln ^{+}\left\vert f\left( x^{0},x^{1}\right) \right\vert d\mu \left(
x^{0},x^{1}\right) . \end{eqnarray*} This completes the proof. \end{proof}

\section{Proof of Theorem \protect\ref{mt2}}

To establish the proof of
Theorem \ref{mt2}, we first provide an auxiliary fact. \begin{lemma} \label{%
Unm} Let $\{\xi _{k,n}^{i}\}_{k,n\in \mathbb{P}}$ $\left( i=0,1\right) $ be
two sequences such that \begin{itemize} \item[(a)] $\sup\limits_{n\in
\mathbb{P}}\sum\limits_{s=1}^{n}\sum\limits_{k=1}^{s}\left\vert \xi
_{k,n}^{i}\right\vert <\infty $, \item[(b)] $\lim\limits_{n\rightarrow
\infty }\xi _{k,n}^{i}=0$, for each fixed $k$. \end{itemize} Assume that $%
\{U_{n^{0},n^{1}}\}_{n^{0},n^{1}\in \mathbb{P}}$ is bounded and \begin{%
equation*} \lim_{\min \left\{ {n}^{0}{,n}^{1}\right\} \rightarrow \infty
}U_{n^{0},n^{1}}=0. \end{equation*}
Then \begin{equation} \lim_{\min \left\{
{n}^{0}{,n}^{1}\right\} \rightarrow \infty
}\sum\limits_{s^{0}=1}^{n^{0}}\sum\limits_{s^{1}=1}^{n^{1}}\varepsilon
_{s^{0}}\left( n^{0}\right) \varepsilon _{s^{1}}\left( n^{1}\right)
\sum\limits_{k^{0}=1}^{s^{0}}\sum\limits_{k^{1}=1}^{s^{1}}\xi
_{k^{0},n^{0}}^{0}\xi _{k^{1},n^{1}}^{1}U_{k^{0},n^{1}}=0 \label{zero} \end{%
equation} \end{lemma}
The proof is evident.

From the definitions of \eqref{%
Wnm},\eqref{H2} and \eqref{H1}, we immediately obtain the following
estimations \begin{equation} \left( \left\vert f\right\vert \ast
D_{2^{s^{0}}}\otimes D_{2^{s^{1}}}\right) \left( x^{0},x^{1}\right) \leq
W_{s^{0},s^{1}}f\left( x^{0},x^{1}\right) , \label{wl-1} \end{equation}%
\begin{equation} W_{s^{0},s^{1}}f\left( x^{0},x^{1}\right) \leq
2^{s^{0}}H_{s^{1}}^{\left( 1\right) }f\left( x^{0},x^{1}\right) , \label{wl-2%
} \end{equation}\begin{equation} W_{s^{0},s^{1}}f\left( x^{0},x^{1}\right)
\leq 2^{s^{1}}H_{s^{0}}^{\left( 0\right) }f\left( x^{0},x^{1}\right) \label{%
wl-3} \end{equation}and\begin{equation} W_{s^{0},s^{1}}f\left(
x^{0},x^{1}\right) \leq c\left( x^{0},x^{1}\right) 2^{s^{0}+s^{1}}. \label{%
wl-4} \end{equation}

\begin{proof}[Proof of Theorem \protect\ref{mt2}]
Assume that $f\in L\ln L(\mathbb{I}^{2})$. Let $%
(x^{0},x^{1})\in \mathbb{I}^{2}$ be a Walsh-Lebesgue point for $f$. We are
going to establish that \begin{equation*} \lim_{\min \left\{
{n_{a}^{0},n_{b}^{1}}\right\} \rightarrow \infty }((\mathcal{T}_{n_{a}^{0}}^{%
\mathbb{T}_{0}}\otimes \mathcal{T}_{n_{b}^{1}}^{\mathbb{T}%
_{1}})f)(x^{0},x^{1})=f(x^{0},x^{1}) \end{equation*}To do so, we denote $%
f_{x^{0},x^{1}}\left( u^{0},u^{1}\right) :=f\left( x^{0},x^{1}\right)
-f\left( u^{0},u^{1}\right) $. Then, it is enough to prove \begin{equation*}
(\mathcal{T}_{n_{a}^{0}}^{\mathbb{T}_{0}}\otimes \mathcal{T}_{n_{b}^{1}}^{%
\mathbb{T}_{1}})f_{x^{0},x^{1}}\rightarrow 0\ \ \text{as}\ \ \min \left\{
{n_{a}^{0},n_{b}^{1}}\right\} \rightarrow \infty . \end{equation*}By \eqref{%
tens} and \eqref{tens1}, keeping in mind \eqref{V1+V2}, we have \begin{%
eqnarray} (\mathcal{T}_{n_{a}^{0}}^{\mathbb{T}_{0}}\otimes \mathcal{T}%
_{n_{b}^{1}}^{\mathbb{T}_{1}})f_{x^{0},x^{1}} &=&f_{x^{0},x^{1}}\ast \left(
V_{n_{a}^{0}}^{\mathbb{T}_{0}}\otimes V_{n_{b}^{1}}^{\mathbb{T}_{1}}\right)
\notag \\[2mm] &=&f_{x^{0},x^{1}}\ast \left( V_{n_{a}^{0},1}^{\mathbb{T}%
_{0}}\otimes V_{n_{b}^{1},1}^{\mathbb{T}_{1}}\right) +f_{x^{0},x^{1}}\ast
\left( V_{n_{a}^{0},1}^{\mathbb{T}_{0}}\otimes V_{n_{b}^{1},2}^{\mathbb{T}%
_{1}}\right) \notag \\ &&+f_{x^{0},x^{1}}\ast \left( V_{n_{a}^{0},2}^{%
\mathbb{T}_{0}}\otimes V_{n_{b}^{1},1}^{\mathbb{T}_{1}}\right)
+f_{x^{0},x^{1}}\ast \left( V_{n_{a}^{0},2}^{\mathbb{T}_{1}}\otimes
V_{n_{b}^{1},2}^{\mathbb{T}_{2}}\right) . \label{conv=1-4} \end{eqnarray}
The equation (\ref{vv}) yields {\small \begin{eqnarray} &&\left\vert
f_{x^{0},x^{1}}\ast \left( V_{n_{a}^{0},1}^{\mathbb{T}_{0}}\otimes
V_{n_{b}^{1},1}^{\mathbb{T}_{1}}\right) \left( x,y\right) \right\vert \label{%
j1-j4} \\ &\lesssim &\underbrace{\sum\limits_{s^{0}=1}^{\left\vert
n_{a}^{0}\right\vert +1}\tilde{\varepsilon}_{s^{0}}(n_{a}^{0})\tau
_{2^{s^{0}},n_{a}^{0}}^{0}\sum\limits_{s^{1}=1}^{\left\vert
n_{b}^{1}\right\vert }\tilde{\varepsilon}_{s^{1}}(n_{b}^{1})\tau
_{2^{s^{1}},n_{b}^{1}}^{1}\left( \left\vert f_{x^{0},x^{1}}\right\vert \ast
D_{2^{s^{0}}}\otimes D_{2^{s^{1}}}\right) \left( x^{0},x^{1}\right) }%
_{J_{1}\left( x^{0},x^{1}\right) } \notag \\[2mm] &&+\underbrace{%
\sum\limits_{s^{0}=1}^{\left\vert n_{a}^{0}\right\vert +1}\tilde{\varepsilon}%
_{s^{0}}(n_{a}^{0})\tau
_{2^{s^{0}},n_{a}^{0}}^{0}\sum\limits_{s^{1}=0}^{\left\vert
n_{b}^{1}\right\vert }2^{s^{1}}\tau _{2^{s^{1}},n_{b}^{1}}^{1}\left(
\left\vert f_{x^{0},x^{1}}\right\vert \ast D_{2^{s^{0}}}\otimes
D_{2^{s^{1}}}\right) \left( x^{0},x^{1}\right) }_{J_{2}\left(
x^{0},x^{1}\right) } \notag \\[2mm] &&+\underbrace{\sum\limits_{s^{0}=1}^{%
\left\vert n_{a}^{0}\right\vert }2^{s^{0}}\tau
_{2^{s^{0}},n_{a}^{0}}^{0}\sum\limits_{s^{1}=1}^{\left\vert
n_{b}^{1}\right\vert +1}\tilde{\varepsilon}_{s^{1}}(n_{b}^{1})\tau
_{2^{s^{1}},n_{b}^{1}}^{1}\left( \left\vert f_{x^{0},x^{1}}\right\vert \ast
D_{2^{s^{0}}}\otimes D_{2^{s^{1}}}\right) \left( x^{0},x^{1}\right) }%
_{J_{3}\left( x^{0},x^{1}\right) } \notag \\[2mm] &&+\underbrace{%
\sum\limits_{s^{0}=0}^{\left\vert n_{a}^{0}\right\vert }2^{s^{0}}\tau
_{2^{s^{0}},n_{a}^{0}}^{0}\sum\limits_{s^{1}=0}^{\left\vert
n_{b}^{1}\right\vert }2^{s^{1}}\tau _{2^{s^{1}},n_{b}^{1}}^{1}\left(
\left\vert f_{x^{0},x^{1}}\right\vert \ast D_{2^{s^{0}}}\otimes
D_{2^{s^{1}}}\right) \left( x^{0},x^{1}\right) }_{J_{4}\left(
x^{0},x^{1}\right) }. \notag \end{eqnarray}}{\normalsize where \begin{%
equation*} \tilde{\varepsilon}_{m}(n)=|\varepsilon _{m-1}(n)-\varepsilon
._{m}(n)| \end{equation*}} {\normalsize Set $\omega \left( t\right)
:=\left\lfloor \frac{1}{4}\log _{2}\left( 1/t\right) \right\rfloor ,t\in
\left( 0,1\right) $. } {\normalsize Now, let us represent $J_{1}\left(
x,y\right) $ as a sum of four terms }{\small \begin{eqnarray} &&J_{1}\left(
x^{0},x^{1}\right) =\left( \sum\limits_{s^{0}=1}^{\omega \left( \left[
t_{0,n_{a}^{0}}^{0}\right] \right) -1}\sum\limits_{s^{1}=1}^{\omega \left( %
\left[ t_{0,n_{b}^{1}}^{1}\right] \right) -1}+\sum\limits_{s^{0}=\omega
\left( \left[ t_{0,n_{a}^{0}}^{0}\right] \right) }^{\left\vert
n_{a}^{0}\right\vert +1}\sum\limits_{s^{1}=0}^{\omega \left( \left[
t_{0,n_{b}^{1}}^{1}\right] \right) -1}\right. \label{J1} \\ &&\left.
+\sum\limits_{s^{0}=1}^{\omega \left( \left[ t_{0,n_{a}^{0}}^{0}\right]
\right) -1}\sum\limits_{s^{1}=\omega \left( \left[ t_{0,n_{b}^{1}}^{1}\right]
\right) }^{\left\vert n_{b}^{1}\right\vert +1}+\sum\limits_{s^{0}=\omega
\left( \left[ t_{0,n_{a}^{0}}^{0}\right] \right) }^{\left\vert
n_{a}^{0}\right\vert +1}\sum\limits_{s^{1}=\omega \left( \left[
t_{0,n_{b}^{1}}^{1}\right] \right) }^{\left\vert n_{b}^{1}\right\vert
+1}\right) \notag \\ &&\left( \tilde{\varepsilon}_{s^{0}}(n_{a}^{0})\tilde{%
\varepsilon}_{s^{1}}(n_{b}^{1})\tau _{2^{s^{0}},n_{a}^{0}}^{0}\tau
_{2^{s^{1}},n_{b}^{1}}^{1}\left( \left\vert f_{x^{0},x^{1}}\right\vert \ast
D_{2^{s^{0}}}\otimes D_{2^{s^{1}}}\right) \left( x^{0},x^{1}\right) \right)
\notag \\ &=&J_{11}\left( x^{0},x^{1}\right) +J_{12}\left(
x^{0},x^{1}\right) +J_{13}\left( x^{0},x^{1}\right) +J_{14}\left(
x^{0},x^{1}\right) . \notag \end{eqnarray}}{\normalsize Due to the condition
of the theorem, one gets }{\small \begin{eqnarray} J_{11}\left( x,y\right)
&\leq &c\left( x^{0},x^{1}\right) \sum\limits_{s^{0}=0}^{\omega \left( \left[
t_{0,n_{a}^{0}}^{0}\right] \right) -1}\sum\limits_{s^{1}=0}^{\omega \left( %
\left[ t_{0,n_{b}^{1}}^{1}\right] \right) -1}2^{s^{0}}\tau
_{2^{s^{0}},n_{a}^{0}}2^{s^{1}}\tau _{2^{s^{1}},n_{b}^{1}} \label{J11} \\
&\leq &c\left( x^{0},x^{1}\right)
t_{0,n_{a}^{0}}^{0}t_{0,n_{b}^{1}}^{1}2^{2\omega \left( \left[
t_{0,n_{a}^{0}}^{0}\right] \right) +2\omega \left( \left[ t_{0,n_{b}^{1}}^{1}%
\right] \right) } \notag \\ &\leq &c\left( x^{0},x^{1}\right) \sqrt{%
t_{0,n_{a}^{0}}^{0}}\sqrt{t_{0,n_{b}^{1}}^{1}}=o\left( 1\right) \ \ \ \text{%
as}\ \ \min \left\{n_{a}^{0},n_{b}^{1}\right\} \rightarrow \infty . \notag
\end{eqnarray}} {\normalsize By \eqref{wl-3}), \eqref{wl1} together with the
conditions of the theorem, we obtain }{\small \begin{eqnarray} J_{12}\left(
x^{0},x^{1}\right) &\leq &c\left( x^{0},x^{1}\right)
\sum\limits_{s^{0}=\omega \left( \left[ t_{0,n_{a}^{0}}^{0}\right] \right)
}^{\left\vert n_{a}^{0}\right\vert +1}\sum\limits_{s^{1}=0}^{\omega \left( %
\left[ t_{0,n_{b}^{1}}^{1}\right] \right) -1}\tilde{\varepsilon}%
_{s^{0}}(n_{a}^{0})\tau _{2^{s^{0}},n_{a}^{0}}^{0}\tau
_{2^{s^{1}},n_{b}^{1}}^{1}W_{s^{0},s^{1}}f\left( x^{0},x^{1}\right) \label{%
J12} \\ &\leq &c\left( x^{0},x^{1}\right) \sum\limits_{s^{0}=\omega \left( %
\left[ t_{0,n_{a}^{0}}^{0}\right] \right) }^{\left\vert n_{a}^{0}\right\vert
+1}\sum\limits_{s^{1}=0}^{\omega \left( \left[ t_{0,n_{b}^{1}}^{1}\right]
\right) -1}\tilde{\varepsilon}_{s^{0}}(n_{a}^{0})\tau
_{2^{s^{0}},n_{a}^{0}}^{0}\tau
_{2^{s^{1}},n_{b}^{1}}2^{s^{1}}H_{s^{0}}^{\left( 0\right) }f\left(
x^{0},x^{1}\right) \notag \\ &\leq &c\left( x^{0},x^{1}\right) \left(
\sum\limits_{s^{0}=\omega \left( \left[ t_{0,n_{a}^{0}}^{0}\right] \right)
}^{\left\vert n_{a}^{0}\right\vert +1}\tilde{\varepsilon}_{s^{0}}(n_{a}^{0})%
\tau _{2^{s^{0}},n_{a}^{0}}^{0}\right) t_{0,n_{b}^{1}}^{1}2^{2\omega \left( %
\left[ t_{0,n_{b}^{1}}^{1}\right] \right) } \notag \\ &\leq &c\left(
x^{0},x^{1}\right) \upsilon (n_{a}^{0},\mathbb{T}^{0})\sqrt{%
t_{0,n_{b}^{1}}^{1}}=o\left( 1\right) \text{ as }\min \left\{
n_{a}^{0},n_{b}^{1}\right\} \rightarrow \infty . \notag \end{eqnarray}%
}{\normalsize Using the same argument, one finds \begin{equation}
J_{13}\left( x^{0},x^{1}\right) =o\left( 1\right) \text{ as }\min \left\{
n_{a}^{0},n_{b}^{1}\right\} \rightarrow \infty . \label{J13} \end{equation}%
Finally, let us examine $J_{14}\left( x,y\right) $. } {\normalsize By
employing \eqref{wl-1} and condition (iii) of Theorem GM, one gets}{\small
\begin{equation*} J_{14}\left( x^{0},x^{1}\right) \leq
\sum\limits_{s^{0}=\omega \left( \left[ t_{0,n_{a}^{0}}^{0}\right] \right)
}^{\left\vert n_{a}^{0}\right\vert +1}\sum\limits_{s^{1}=\omega \left( \left[
t_{0,n_{b}^{1}}^{1}\right] \right) }^{\left\vert n_{b}^{1}\right\vert +1}%
\tilde{\varepsilon}_{s^{0}}(n_{a}^{0})\tilde{\varepsilon}_{s^{1}}(n_{b}^{1})%
\tau _{2^{s^{0}},n_{a}^{0}}^{0}\tau
_{2^{s^{1}},n_{b}^{1}}^{1}W_{s^{0},s^{1}}f\left( x^{0},x^{1}\right) =o\left(
1\right) \end{equation*}}{\normalsize as $\min \left\{
n_{a}^{0},n_{b}^{1}\right\} \rightarrow \infty $. Here we have used the fact
that $\left( x^{0},x^{1}\right) $ is a Lebesgue point and hence $%
W_{s^{0},s^{1}}f\left( x^{0},x^{1}\right) $ approaches to zero as $\min
\left\{ s^{0},s^{1}\right\} \rightarrow \infty $. } {\normalsize Now,
combining the last one together with (\ref{j1-j4}), (\ref{J1}), (\ref{J11}),
(\ref{J12}), (\ref{J13}) we arrive at \begin{equation} \lim_{\min \left\{
n_{a}^{0},n_{b}^{1}\right\} \rightarrow \infty }J_{1}\left(
x^{0},x^{1}\right) =0. \label{J1=0} \end{equation}} {\normalsize
Consequently, using the same technique one can establish the followings: \begin{%
equation*} \lim_{\min \left\{ n_{a}^{0},n_{b}^{1}\right\} \rightarrow \infty
}J_{k}\left( x^{0},x^{1}\right) =0,\ \ \ k=2,3,4. \end{equation*}}
{\normalsize We set \begin{equation*} \widetilde{t}_{n_{b}\left( s-1\right)
+k,n_{b}}=\left\{ \begin{array}{c} t_{n_{b}\left( s-1\right) +k,n_{b}},\ \
k<2^{s} \\ 0,\ \ k=2^{s}\end{array},\ \ s\in \mathbb{P}.\right. \end{%
equation*}Then, (\ref{V12}) may be estimated as follows }{\small \begin{%
equation} \left\vert V_{n_{b},2}^{\mathbb{T}_{1}}\right\vert \leq
\sum\limits_{s=0}^{|n_{b}|}\varepsilon _{s}\left( n_{b}\right) \left(
\sum\limits_{l=1}^{2^{s}-1}\left( \widetilde{t}_{n_{b}\left( s-1\right)
+l,n_{b}}-\widetilde{t}_{n_{b}\left( s-1\right) +l+1,n_{b}}\right)
l\left\vert K_{l}\right\vert \right) . \label{up1} \end{equation}%
}{\normalsize Consequently}, {\small \begin{eqnarray*} \left\vert
f_{x^{0}x^{1}}\ast \left( V_{n_{a}^{0},2}^{\mathbb{T}_{0}}\otimes
V_{n_{b}^{1},2}^{\mathbb{T}_{1}}\right) \right\vert &\leq
&\sum\limits_{s^{0}=0}^{|n_{a}^{0}|}\varepsilon _{s^{0}}\left(
n_{a}^{0}\right) \sum\limits_{s^{1}=0}^{|n_{b}^{1}|}\varepsilon
_{s^{1}}\left( n_{b}^{1}\right)
\sum\limits_{l^{0}=1}^{2^{s^{0}}-1}\sum%
\limits_{l^{1}=1}^{2^{s^{1}}-1}l^{0}l^{1} \\ &&\times \left( \widetilde{t}%
_{n_{a}^{0}\left( s^{0}-1\right) +l^{0},n_{a}^{0}}^{0}-\widetilde{t}%
_{n_{a}^{0}\left( s^{1}-1\right) +l^{0}+1,n_{a}^{0}}^{0}\right) \\ &&\times
\left( \widetilde{t}_{n_{b}\left( s^{1}-1\right) +l^{1},n_{b}}^{1}-%
\widetilde{t}_{n_{b}\left( s^{1}-1\right) +l^{1}+1,n_{b}}^{1}\right) \\
&&\times \left( \left\vert f_{x^{0}x^{1}}\right\vert \ast \left\vert
K_{l^{0}}\right\vert \otimes \left\vert K_{l^{1}}\right\vert \right) . \end{%
eqnarray*}}

Now, we define the following double sequence \begin{equation*}
U_{l^{0},l^{1}}\left( x^{0},x^{1}\right) :=\left\vert
f_{x^{0}x^{1}}\right\vert \ast \left\vert K_{l^{0}}\right\vert \otimes
\left\vert K_{l^{1}}\right\vert . \end{equation*}One can see that it is
bounded. Indeed, by the result of \cite{WeiszWLP3} we have \begin{equation}
\left\vert f_{x^{0}x^{1}}\right\vert \ast \left( \left\vert
K_{l^{0}}\right\vert \otimes \left\vert K_{l^{1}}\right\vert \right) \leq
\frac{1}{l^{0}l^{1}}\sum\limits_{i^{0}=0}^{\left\vert l^{0}\right\vert
}\sum\limits_{i^{1}=0}^{\left\vert l^{1}\right\vert
}2^{i^{0}+i^{1}}W_{i^{0},i^{1}}f\left( x^{0},x^{1}\right) . \label{zz}
\end{equation}
Take a natural number $M:=M\left( x^{0},x^{1}\right) $ such that
\begin{equation*} W_{i^{0},i^{1}}f\left( x^{0},x^{1}\right) <1,\ \
i^{0},i^{1}>M. \end{equation*}Then, by \eqref{wl2}, \eqref{wl3} and \eqref{zz%
}, one finds
\small
\begin{eqnarray*} U_{l^{0},l^{1}}\left( x^{0},x^{1}\right)
&\leq &\frac{1}{l^{0}l^{1}}\sum\limits_{0\leq i^{0}\leq M}\sum\limits_{0\leq
i^{1}\leq \left\vert l^{1}\right\vert }2^{i^{0}+i^{1}}W_{i^{0},i^{1}}f\left(
x^{0},x^{1}\right) \\ &&+\frac{1}{l^{0}l^{1}}\sum\limits_{0\leq i^{0}\leq
\left\vert l^{0}\right\vert }\sum\limits_{0\leq i^{1}\leq
M}2^{i^{0}+i^{1}}W_{i^{0},i^{1}}f\left( x^{0},x^{1}\right) \\ &&+\frac{1}{%
l^{0}l^{1}}\sum\limits_{M<i^{0}\leq \left\vert l^{0}\right\vert
}\sum\limits_{M<i^{1}\leq \left\vert l^{1}\right\vert
}2^{i^{0}+i^{1}}W_{i^{0},i^{1}}f\left( x^{0},x^{1}\right) \\ &\leq &\frac{1}{%
l^{0}l^{1}}\sum\limits_{0\leq i^{0}\leq M}\sum\limits_{0\leq i^{1}\leq
\left\vert l^{1}\right\vert }2^{2i^{0}+i^{1}}H_{i^{1}}^{\left( 2\right)
}f\left( x^{0},x^{1}\right) \\ &&+\frac{1}{l^{0}l^{1}}\sum\limits_{0\leq
i^{0}\leq \left\vert l^{0}\right\vert }\sum\limits_{0\leq i^{1}\leq
M}2^{i^{0}+2i^{1}}H_{i^{0}}^{\left( 1\right) }f\left( x^{0},x^{1}\right) \\
&&+\frac{1}{l^{0}l^{1}}\sum\limits_{M<i^{0}\leq \left\vert l^{0}\right\vert
}\sum\limits_{M<i^{1}\leq \left\vert l^{1}\right\vert
}2^{i^{0}+i^{1}}W_{i^{0},i^{1}}f\left( x^{0},x^{1}\right) \\ &\leq &C\left(
x^{0},x^{1}\right) <\infty . \end{eqnarray*}
\normalsize
Now, we are going to establish
\begin{equation} U_{l^{0},l^{1}}\left( x^{0},x^{1}\right) =o\left( 1\right)
\text{ \ as \ }\min \left\{ l^{0},l^{1}\right\} \rightarrow \infty . \label{%
00} \end{equation}Indeed, again employing the estimations (\ref{wl2}), (\ref{%
wl3}) and (\ref{GG}) we obtain
\small
\begin{eqnarray*} \left\vert
f_{x^{0}x^{1}}\right\vert \ast \left( \left\vert K_{l^{0}}\right\vert
\otimes \left\vert K_{l^{1}}\right\vert \right) &\leq &\frac{1}{l^{0}l^{1}}%
\sum\limits_{0\leq i^{0}\leq \omega \left( 1/l^{0}\right)
}\sum\limits_{0\leq i^{1}\leq \omega \left( 1/l^{1}\right)
}2^{2i^{0}+2i^{1}} \\ &&+\frac{1}{l^{0}l^{1}}\sum\limits_{0\leq i^{0}\leq
\omega \left( 1/l^{0}\right) }\sum\limits_{\omega \left( 1/l^{1}\right)
<i^{1}\leq l^{1}}2^{2i^{0}+i^{1}}H_{i^{1}}^{\left( 1\right) }f\left(
x^{0},x^{1}\right) \\ &&+\frac{1}{l^{0}l^{1}}\sum\limits_{\omega \left(
1/l^{0}\right) <i^{0}\leq l^{0}}\sum\limits_{0\leq i^{1}\leq \omega \left(
1/l^{1}\right) }2^{i^{0}+2i^{1}}H_{i^{0}}^{\left( 0\right) }f\left(
x^{0},x^{1}\right) \\ &&+\frac{1}{l^{0}l^{1}}\sum\limits_{\omega \left(
1/l^{0}\right) <i^{0}\leq l^{0}}\sum\limits_{\omega \left( 1/l^{1}\right)
<i^{1}\leq l^{1}}W_{i^{0},i^{1}}f\left( x^{0},x^{1}\right) \\ &\leq &C\left(
x^{0},x^{1}\right) \left\{ \frac{1}{\sqrt{l^{0}}}+\frac{1}{\sqrt{l^{1}}}%
+\sup\limits_{i^{s}>\delta \left( i^{s}\right) ,s=0,1}W_{i^{0},i^{1}}f\left(
x^{0},x^{1}\right) \right\} \\ &=&o\left( 1\right) \text{ as }\min \left\{
l^{0},l^{1}\right\} \rightarrow \infty . \end{eqnarray*}
\normalsize
It remains to
demonstrate that conditions (a) and (b) are correct. The validity of
condition (b) is simply checked: \begin{equation} t_{k,n_{a}^{j}}^{j}\leq
t_{0,n_{a}^{j}}^{j}=o\left( 1\right) \text{ as }n_{a}^{j}\rightarrow \infty
\left( j=0,1\right) . \label{t1} \end{equation}

Denote
\begin{equation*} \xi
_{s^{0},n^{0}}^{0}:=\sum\limits_{l^{0}=1}^{2^{s^{0}}-1}l^{0}\left(
\widetilde{t}_{n_{a}^{0}\left( s^{0}-1\right) +l^{0},n_{a}^{0}}^{0}-%
\widetilde{t}_{n_{a}^{0}\left( s^{1}-1\right) +l^{0}+1,n_{a}^{0}}^{0}\right)
. \end{equation*}
\begin{equation*} \xi
_{s^{1},n^{1}}^{1}:=\sum\limits_{l^{1}=1}^{2^{s^{1}}-1}l^{0}\left(
\widetilde{t}_{n_{b}^{1}\left( s^{1}-1\right) +l^{1},n_{b}^{1}}^{1}-%
\widetilde{t}_{n_{b}^{1}\left( s^{1}-1\right) +l^{1}+1,n_{b}^{1}}^{1}\right)
. \end{equation*}
Regarding (a), we have
\small
\begin{eqnarray}
\sum\limits_{s^{0}=1}^{|n_{a}^{0}|}\varepsilon _{s^{0}}\left(
n_{a}^{0}\right) \xi _{s^{0},n^{0}}^{0} \label{t2}
&=&\sum\limits_{s^{0}=1}^{|n_{a}^{0}|}\varepsilon _{s^{0}}\left(
n_{a}^{0}\right) \left\{ \sum\limits_{l^{0}=1}^{2^{s^{0}}-2}l^{0}\left(
t_{n_{a}^{0}\left( s^{0}-1\right) +l^{0},n_{a}^{0}}^{0}-t_{n_{a}^{0}\left(
s^{1}-1\right) +l^{0}+1,n_{a}^{0}}^{0}\right) \right. \notag \\ &&\left.
+\left( 2^{s^{0}}-1\right) t_{n_{a}^{0}\left( s^{0}-1\right)
+2^{s^{0}}-1,n_{a}^{0}}^{0}\right\} \notag \\
&=&\sum\limits_{s^{0}=1}^{|n_{a}^{0}|}\varepsilon _{s^{0}}\left(
n_{a}^{0}\right) \left\{ \sum\limits_{l^{0}=1}^{2^{s^{0}}-2}\left(
l^{0}t_{n_{a}^{0}\left( s^{0}-1\right) +l^{0},n_{a}^{0}}^{0}-\left(
l^{0}+1\right) t_{n_{a}^{0}\left( s^{1}-1\right)
+l^{0}+1,n_{a}^{0}}^{0}\right) \right. \notag \\ &&\left.
+\sum\limits_{l^{0}=1}^{2^{s^{0}}-2}t_{n_{a}^{0}\left( s^{1}-1\right)
+l^{0}+1,n_{a}^{0}}^{0}+\left( 2^{s^{0}}-1\right) t_{n_{a}^{0}\left(
s^{0}-1\right) +2^{s^{0}}-1,n_{a}^{0}}^{0}\right\} \notag \\
&=&\sum\limits_{s^{0}=1}^{|n_{a}^{0}|}\varepsilon _{s^{0}}\left(
n_{a}^{0}\right) \left(
\sum\limits_{l^{0}=1}^{2^{s^{0}}-2}t_{n_{a}^{0}\left( s^{0}-1\right)
+l^{0}+1,n_{a}^{0}}^{0}\right) \notag \\ &\leq
&\sum\limits_{s^{0}=1}^{|n_{a}^{0}|}\varepsilon _{s^{0}}\left(
n_{a}^{0}\right) \left( \tau _{2^{s^{0}},n_{a}^{0}}-\tau
_{2^{s^{0}-1},n_{a}^{0}}\right) +\tau _{1,n_{a}^{0}}^{0} \notag \\ &\leq
&\sum\limits_{s^{0}=1}^{|n_{a}^{0}|-1}\left\vert \varepsilon _{s^{0}}\left(
n_{a}^{0}\right) -\varepsilon _{s^{0}+1}\left( n_{a}^{0}\right) \right\vert
\tau _{2^{s^{0}},n_{a}^{0}}^{0}\leq c<\infty . \notag \end{eqnarray}
\normalsize
Analogously we can prove the same for $\left\{
\xi_{l^{1},n_{b}^{1}}^{1}:l^{1},n_{b}^{1}\in \mathbb{N}\right\} .$

Now, by Lemma \ref{Unm} we infer that
\begin{equation} \left\vert
f_{x^{0}x^{1}}\ast \left( V_{n_{a}^{0},2}^{\mathbb{T}_{0}}\otimes
V_{n_{b}^{1},2}^{\mathbb{T}_{1}}\right) \right\vert =o\left( 1\right) \text{
as }\min \left\{ n_{a}^{0},n_{b}^{1}\right\} \rightarrow \infty . \label{22}
\end{equation}

Now, let us consider
\small
\begin{eqnarray*} \left\vert
f_{x^{0}x^{1}}\ast \left( V_{n_{a}^{0},1}^{\mathbb{T}_{0}}\otimes
V_{n_{b}^{1},2}^{\mathbb{T}_{1}}\right) \right\vert &\leq &\left(
\sum\limits_{s^{0}=0}^{\left\vert n_{a}^{0}\right\vert +1}\tilde{\varepsilon}%
_{s^{0}}(n_{a}^{0})\tau _{2^{s^{0}},n_{a}^{0}}^{0}\right) \\[2mm] &&\times
\left( \sum\limits_{s^{1}=0}^{|n_{b}^{1}|}\varepsilon _{s^{1}}\left(
n_{b}^{1}\right) \xi _{s^{1},n^{1}}^{1}\right) \left( \left\vert
f_{x^{0}x^{1}}\right\vert \ast \left( D_{2^{s^{0}}}\otimes \left\vert
K_{l^{1}}\right\vert \right) \right) \\ &&+\left(
\sum\limits_{s^{0}=1}^{\left\vert n_{a}^{0}\right\vert }2^{s^{0}}\tau
_{2^{s^{0}},n_{a}^{0}}^{0}\right) \left(
\sum\limits_{l^{1}=1}^{2^{s^{1}}-1}\xi _{s^{1},n^{1}}^{1}\right) \left(
\left\vert f_{x^{0}x^{1}}\right\vert \ast \left( D_{2^{s^{0}}}\otimes
\left\vert K_{l^{1}}\right\vert \right) \right) .
\end{eqnarray*}
\normalsize
Since\begin{equation*} \left\vert f_{x^{0}x^{1}}\right\vert \ast \left(
D_{2^{s^{0}}}\otimes \left\vert K_{l^{1}}\right\vert \right) \leq \frac{%
2^{s^{0}}}{l^{1}}\sum\limits_{i^{1}=0}^{\left\vert l^{1}\right\vert
}2^{i^{1}}W_{2^{s^{0}},i^{1}}f\left( x^{0},x^{1}\right) \end{equation*}
and%
\begin{equation*} \sum\limits_{s^{0}=1}^{\left\vert n_{a}^{0}\right\vert
}2^{s^{0}}t_{2^{s^{0}},n_{a}^{0}}^{0}\leq 2\tau _{2^{\left\vert
n_{a}^{0}\right\vert },n_{a}^{0}}\leq 2 \end{equation*} together with (\ref{%
zero}), (\ref{t1}) we arrive at
\begin{equation*} \left\vert
f_{x^{0}x^{1}}\ast \left( V_{n_{a}^{0},1}^{\mathbb{T}_{0}}\otimes
V_{n_{b}^{1},2}^{\mathbb{T}_{1}}\right) \right\vert =o\left( 1\right) \text{
as }\min \left\{ n_{a}^{0},n_{b}^{1}\right\} \rightarrow \infty . \label{21}
\end{equation*}

By
the same argument one establish that \begin{equation*} \left\vert
f_{x^{0}x^{1}}\ast \left( V_{n_{a}^{0},2}^{\mathbb{T}_{0}}\otimes
V_{n_{b}^{1},1}^{\mathbb{T}_{1}}\right) \right\vert =o\left( 1\right) \text{
as }\min \left\{ n_{a}^{0},n_{b}^{1}\right\} \rightarrow \infty . \label{21}
\end{equation*}
This completes the proof.
\end{proof}

\begin{remark} Take two $\{\alpha _{n}^{(i)}\}\subset\left( 0,1\right]$, ($i=0,1$) sequences. Let $\mathbb{A}^{(i)}=(t_{k,n}^{i})$ be the corresponding matrix of transformations, where
$$t_{k,n}^{i}:=\frac{A_{k}^{\alpha _{n}^{i}-1}}{%
A_{n}^{\alpha _{n}^{i}}}.$$
Now, we consider the associated $(C,\alpha_n)$-means, i.e.
$$\mathcal{T}_{n^{i}}^{\mathbb{A}^{(i)}}\left( f\right) :=\sigma
_{n^{i}}^{\alpha _{n}^{i}}\left( f\right) =f\ast V_{n^{i}}^{\mathbb{A}^{(i)}},
$$
where
$$
V_{n^{i}}^{\mathbb{A}^{(i)}}=\frac{1}{A_{n^{i}}^{\alpha _{n}^{i}}}\sum%
\limits_{k=1}^{n^{i}}A_{n^{i}-k}^{\alpha _{n}^{i}-1}D_{k}.
$$
Let us define two-dimensional
operators resulting from the tensor product:
\begin{equation*}
\mathfrak{T}^{\mathbb{A}^{(0)},\mathbb{A}^{(1)}}f_{n^{0},n^{1}}f:=f\ast \left( V_{n^{0}}^{\mathbb{A}^{(0)}}\otimes V_{n^{1}}^{\mathbb{A}^{(1)}}\right) .
\end{equation*}
Assume that the sequences $\left(
\alpha _{n}^{i}:n^{i}\in N\right) ,i=0,1$ are stationary, i. e. $\alpha
_{n}^{i}=\alpha ^{i}>0.$ Then, by the result of \cite{Weisz-book2} one finds
\begin{equation}
\sup_{n}\left\Vert \mathcal{T}_{n^{i}}^{\mathbb{A}^{(i)}}\right\Vert _{L_{\infty }\left( \mathbb{I}\right) \rightarrow L_{\infty }
\left( \mathbb{I
}\right) }<\infty; \ \ i=0,1\label{binf}
\end{equation}
and consequently, according of Theorem \ref{mt1} the sequence $\{\mathfrak{T}^{\mathbb{A}^{(0)},\mathbb{A}^{(1)}}f\}$ converges almost everywhere to $f$ on $\mathbb{I}^{2}$, for $f\in L\ln L\left( \mathbb{I}^{2}\right)$. Furthermore,  by Theorem \ref{mt2}, we infer that the convergence occurs at every two-dimensional Walsh-Lebesgue point.

Now consider the case when $\alpha _{n^{i}}^{i}\rightarrow 0$ ($i=0,1$). Then, as is known, the inequality (\ref{binf})
does not hold, and therefore Theorems \ref{mt1} and \ref{mt2} are not applicable. 
Moreover, no matter how slowly $\alpha _{n^{i}}^{i}$ tends to zero,
the almost everywhere convergence does not occur even in the case of a sequence
of one-dimensional operators (see \cite{GoginavaJMAA23}, for details). Then it natural to find subsequences $\left( n_{a}^{i}:a\in \mathbb{N}\right) $ along which the
sequence of operators $\left\{\mathcal{T}_{n_a^{i}}^{\mathbb{A}^{(i)}}\right\}$ satisfies (\ref{binf}), and therefore $%
\mathfrak{T}^{\mathbb{A}^{(0)},\mathbb{A}^{(1)}}f_{n^{0}_a,n^{1}_b}f \rightarrow f$ almost everywhere on $\mathbb{I}^{2}$ for $f\in L\ln
L\left( \mathbb{I}^{2}\right)$. We stress that such kinds of subsequences can be found by 
Theorem 1 (iii).  Namely, a sufficient condition for those subsequences is given by 
\begin{equation}
\sup\limits_{a\in \mathbb{N}}\frac{1}{2^{|n_{a}^{i}|\alpha _{n}^{i}}}%
\sum_{k=0}^{|n_{a}^{i}|}|\varepsilon _{k}(n_{a}^{i})-\varepsilon
_{k+1}(n_{a}^{i})|2^{k\alpha _{n}^{i}}<\infty . \label{c2} \end{equation}%
Hence,  the subsequences satisfy condition (\ref{c2}), then the sequence
$\{\mathfrak{T}^{\mathbb{A}^{(0)},\mathbb{A}^{(1)}}_{n^{0}_a,n^{1}_b}f\}$ is convergent to $%
f\in L\ln L\left( \mathbb{I}^{2}\right) $ at every Walsh-Lebesgue point.
\end{remark}

\appendix 
\section{Walsh System} 

In this section, we recal some
necessary denotations about the dyadic analysis. A dyadic interval in $%
\mathbb{I}:=[0,1)$ means an interval in the form $I\left( l,k\right) :=\left[
\frac{l}{2^{k}},\frac{l+1}{2^{k}}\right) $ for some $k\in \mathbb{N}$, $%
0\leq l<2^{k}$. Given $k\in \mathbb{N}$ and $x\in \mathbb{I},$ $I_{k}(x)$
denotes the dyadic interval of length $2^{-k}$ which contains the point $x$.
For the sake of shortness, $I_{n}:=I_{n}\left( 0\right) \left( n\in \mathbb{N%
}\right)$ is denoted as $\overline{I}_{k}\left( x\right) :=I\backslash
I_{k}\left( x\right) $. Given $n\in \mathbb{N}$, $n\neq 0$ by $\left\vert
n\right\vert $ which indicates $2^{\left\vert n\right\vert }\leq
n<2^{\left\vert n\right\vert +1}$. Let\begin{equation*} x=\sum%
\limits_{n=0}^{\infty }x_{n}2^{-\left( n+1\right) } \end{equation*}be the
dyadic expansion of $x\in \mathbb{I}$, where $x_{n}\in \{0,1\}$. If $x $ is
a dyadic rational number the expansion which terminate in $0^{\prime } $s is
chosen. By $\dotplus $, the logical addition on $\mathbb{I}$ is denoted,
i.e. for any $x,y\in \mathbb{I}$ \begin{equation*} x\dotplus
y:=\sum\limits_{n=0}^{\infty }\left\vert x_{n}-y_{n}\right\vert 2^{-\left(
n+1\right) }. \end{equation*} For every $n\in \mathbb{N}$ the following
binary expansion can be written\begin{equation*} n=\sum\limits_{k=0}^{\infty
}\varepsilon _{k}\left( n\right) 2^{k}, \end{equation*}where $\varepsilon
_{k}\left( n\right) =0$ or $1$ for $k\in \mathbb{N}$. The numbers $%
\varepsilon _{k}\left( n\right) $ will be called the binary coefficients of $%
n$. The Rademacher system is defined by \begin{equation*} \rho _{n}\left(
x\right) :=\left( -1\right) ^{x_{n}}\text{ \ \ }\left( x\in \mathbb{I},n\in
\mathbb{N}\right) . \end{equation*} The Walsh-Paley system is defined as the
sequence of the Walsh-Paley functions:\begin{equation*} w_{n}\left( x\right)
:=\prod\limits_{k=0}^{\infty }\left( \rho _{k}\left( x\right) \right)
^{\varepsilon _{k}\left( n\right) }=\left( -1\right)
^{\sum\limits_{k=0}^{\left\vert n\right\vert }\varepsilon _{k}\left(
n\right) x_{k}}\quad \left( x\in \mathbb{I},n\in \mathbb{N}\right) . \end{%
equation*} The Walsh-Dirichlet kernel is defined by \begin{equation*}
D_{n}:=\sum\limits_{k=0}^{n-1}w_{k}\quad \left( n\in \mathbb{N}\right) ,\
D_{0}=0. \end{equation*} It is well-known \cite{GES,SWS} that \begin{equation%
} D_{2^{n}}=2^{n}\mathbf{1}_{I_{n}}. \label{dir2} \end{equation} Given $f\in
L_{1}\left( \mathbb{I}\right) $ its partial sums of the Walsh-Fourier series
are defined by $S_{m}\left( f\right) :=\sum\limits_{i=0}^{m-1}\widehat{f}%
\left( i\right) w_{i}.$ Denote \begin{equation*} E_{n}\left( f\right)
:=S_{2^{n}}\left( f\right) ,\quad E^{\ast }\left( f\right)
:=\sup\limits_{n\in \mathbb{N}}\left\vert E_{n}\left( f\right) \right\vert .
\end{equation*} The dyadic Hardy space $H_{1}(\mathbb{I})$ is the set of all
functions $f\in L_{1}\left( \mathbb{I}\right) $ such that \begin{equation*}
\left\Vert f\right\Vert _{H_{1}}:=\left\Vert E^{\ast }\left( f\right)
\right\Vert _{1}<\infty . \end{equation*} 
 \end{document}